\documentclass{article}

\usepackage{amsmath}
\usepackage{amssymb}
\usepackage{amsthm}
\newtheoremstyle{plain2}{\topsep}{\topsep}%
     {\itshape}
     {}
     {\bfseries}
     {.}
     {.5em}
     {\thmnumber{(#2)}\thmname{ #1}\thmnote{ #3}}

\theoremstyle{plain2}
\newtheorem{teo}{Theorem}[section]
\newtheorem{prop}[teo]{Proposition}
\newtheorem{coro}[teo]{Corollary}
\newtheorem{lemma}[teo]{Lemma}

\newtheoremstyle{definition2}{\topsep}{\topsep}%
     {}
     {}
     {\bfseries}
     {.}
     {.5em}
     {\thmnumber{(#2)}\thmname{ #1}\thmnote{ #3}}

\theoremstyle{definition2}

\newtheorem{rem}[teo]{Remark}

\def\R{\mathbb{R}}

\def\a{\alpha}
\def\b{\beta}
\def\g{\gamma}
\def\d{\delta}
\def\ep{\varepsilon}
\def\eps{\varepsilon}
\def\f{\varphi}
\def\l{\lambda}
\def\o{\omega}

\def\G{\Gamma}
\def\D{\Delta}

\def\O{\Omega}

\def\ub{u_\b}
\def\vb{v_\b}
\def\lb{\l_\b}
\def\mub{\mu_\b}
\def\hb{h_\b}
\def\kb{k_\b}

\def\Omegabar{\overline{\Omega}}
\def\ubar{\bar{u}}
\def\vbar{\bar{v}}
\def\hbar{\bar{h}}
\def\kbar{\bar{k}}

\def\xo{x_0}
\def\fs{|x-\xo|^{N-2}}
\def\nablat{\nabla_\theta}
\def\ur{u_{(r)}}
\def\vr{v_{(r)}}
\def\urtilde{\tilde{u}_{(r_n)}}
\def\vrtilde{\tilde{v}_{(r_n)}}
\def\Ntilde{\tilde{N}}
\def\Eb{E_\b}
\def\Hb{H_\b}
\def\Nb{N_\b}
\def\Rb{R_\b}
\def\Omegatilde{\tilde{\Omega}}
\def\rbar{\bar{r}}

\def\spc{H^1_0(\O)}

\def\dis{\displaystyle}

\newcommand{\parder}[2]{\partial_{ #2} #1}

\title{\sc Uniform H\"older bounds for nonlinear Schr\"odinger systems with strong competition
\footnote{Work partially supported by MIUR, Project ``Metodi
Variazionali ed Equazioni Differenziali Non Lineari'' }}

\author{%
Benedetta Noris, Hugo Tavares, Susanna Terracini and Gianmaria Verzini}

\begin{document}

\maketitle

\begin{abstract}
For the positive solutions of the Gross--Pitaevskii system
\[
\left\{\begin{array}{llll}
-\D \ub + \lb \ub = \o_1 \ub^3 -\b \ub \vb^2 \\
-\D \vb + \mub \vb = \o_2 \vb^3 -\b \ub^2 \vb
\end{array}\right.
\]
we prove that $L^\infty$--boundedness implies $C^{0,\alpha}$--boundedness, uniformly as $\b\to+\infty$, for every $\a\in(0,1)$. Moreover we prove that the limiting profile, as $\b\to+\infty$, is Lipschitz continuous. The proof relies upon the blow--up technique and the monotonicity formulae by Almgren and Alt--Caffarelli--Friedman. This system arises in the Hartree-Fock approximation theory for binary mixtures of Bose--Einstein condensates in different hyperfine states. Extensions to systems with $k>2$ densities are given.

\bigskip


\end{abstract}

\section{Introduction}
The purpose of this paper is to prove uniform bounds in H\"older
norm for families of positive solutions to nonlinear Schr\"odinger equations of the form
\begin{equation} \label{system_premain}
\left\{\begin{array}{llll}
-\D \ub + \lb \ub = \o_1 \ub^3 -\b \ub \vb^2 \\
-\D \vb + \mub \vb = \o_2 \vb^3 -\b \ub^2 \vb \smallskip\\
u_\b,v_\b\in\spc,
  \end{array}\right.
\end{equation}
for the competition parameter $\b\in(0,+\infty)$.
Such systems arise in different physical applications, such as the
determination of standing waves in a binary mixture of Bose-Einstein
condensates in two different hyperfine states. While the sign of the
parameter $\o_i$ discriminates between the focusing and defocusing
behavior of a single component, the sign of $\beta$ determines the
type of interplay between the two states. When positive, the two
states are in competition and repel each other. In this paper we
deal with diverging interspecific competition rates (both in the
focusing and defocusing case). The limiting behavior is known for
the ground state solutions: as $\beta\to+\infty$ the wave amplitudes
segregate, that is, their supports tend to be disjoint. This
phenomenon, called phase separation, has been studied, in the case
of $\o_i> 0$ (focusing), starting from \cite{ctv,ctv2}, and, when
$\o_i<0$ (defocusing), in \cite{clll}. As far as the excited states
are concerned,  the recent literature shows that other families of
solutions exist for large $\beta$'s  (\cite{ww, dww, TV}). The
asymptotic behavior of such families of  solutions has been
investigated in  \cite{ww3}, where, in the case of planar systems,
it is proved uniform convergence to a segregated limiting profile
$(u,v)$, where each component satisfies the equation
\begin{equation}\label{eq:limit_sys}
 \left\{\begin{array}{lll}
-\D u +\l u &=& \o_1 u^3  \quad \text{ in } \{u>0\},\\
  -\D v +\mu v &=& \o_2 v^3  \quad \text{ in } \{v>0\}.
\end{array}\right.
\end{equation}
In this paper we improve the result of \cite{ww3}, proving bounds in
H\"older norms whenever $\O\subset\R^N$ is a smooth bounded domain,
in dimension $N=2,3$ (and also in higher dimension, provided the
cubic nonlinearities are replaced with subcritical ones). Besides
the validity of the equations above, we prove Lipschitz regularity
of the limiting profile. Our result relies upon the blow--up
technique (section \ref{sec:blowup}) and suitable Liouville--type
theorems (section \ref{section_Liouville}). Such a strategy has been
already adopted by some of the authors in \cite{ctv4} in proving
uniform H\"older estimates for competition--diffusion systems with
Lotka--Volterra type of interactions. The arguments there, however,
though helpful in the present situation, need to be complemented
with some new ideas, including a proper use of the Almgren's
frequency formula \cite{Alm}. This requires the systems to have a
gradient form. Let us mention that H\"older estimates for (non
gradient) coupling arising in combustion theory have been obtained
in \cite{cr}. Regularity of the limiting profile and its nodal set,
for ground states and other minimizing vector solutions has been
established in \cite{ctv15,caflin}. Our main results write as
follows.

\begin{teo}\label{theorem_premain}
Let $\ub, \vb$ be positive solutions of \eqref{system_premain} uniformly bounded in $L^\infty(\Omega)$, where $\lb,\,\mub$ are bounded in $\R$ and $\o_1,\o_2$ are fixed constants. Then for every $\a \in (0,1)$ there exists $C>0$, independent of $\b$, such that
\[
\left\|(u_\b,v_\b)\right\|_{C^{0,\a}(\overline{\O})}\leq C\quad\text{for every }\b>0.
\]
\end{teo}

\begin{teo}\label{theorem_lipschitz}
Under the assumptions of the previous theorem, there exists a pair
$(u,v)$ of Lipschitz continuous functions such that, up to a
subsequence, there holds
\begin{itemize}
\item[(i)] $\ub\to u, \vb\to v$ in $C^{0,\a}(\overline{\Omega})\cap H^1(\Omega)$, $\forall \ \a\in (0,1)$;
\item[(ii)] $u\cdot v\equiv 0$ in $\Omega$ and $\displaystyle \int_\Omega \b u_\b^2 v_\b^2\rightarrow 0$ as $\b\rightarrow +\infty$;
\item[(iii)] the limiting functions $u,v$ satisfy system \eqref{eq:limit_sys} with
$\l:=\lim_{\b\to+\infty} \l_\b$, $\mu:=\lim_{\b\to+\infty} \mu_\b$.
\end{itemize}
\end{teo}

For the sake of simplicity we consider here systems of two
components, but all the results extend to the case of systems of $k$
equations
\begin{equation} \label{system_premain_k}
\left\{\begin{array}{llll}
\displaystyle -\D u_i + \lambda_i u_i = \omega_i u^3_i -\b u_i\sum_{i\neq j}\beta_{ij} u^2_j,\quad i=1,\dots,k \smallskip\\
u_i \in \spc,
  \end{array}\right.
\end{equation}
provided that it possesses a gradient structure, i.e.
$\beta_{ij}=\beta_{ji}$ (see Remark \ref{rem:k}).

The study of system \eqref{system_premain} will be carried out as a
particular case of a more general one, where $L^2$--perturbations
are allowed. The reason for this approach is that, in a forthcoming
paper, the authors intend to present a variational construction to
obtain, for every fixed $\b$, several solutions of
\eqref{system_premain}; the present estimates, in their more general
version, will then be used to study how, and in which sense, such a
variational structure passes to the limit as $\b\to+\infty$. To be
more precise, let us consider the system
\begin{equation} \label{system_main}
\left\{\begin{array}{llll}
-\D \ub + \lb \ub = \o_1 \ub^3 -\b \ub \vb^2 +\hb \quad &\text{in }  \Omega\\
-\D \vb + \mub \vb = \o_2 \vb^3 -\b \ub^2 \vb +\kb \quad &\text{in }  \Omega\smallskip\\
\ub,\ \vb \in H^1_0(\Omega), \quad \ub, \vb \geq 0 \quad &\text{in } \Omega
  \end{array}\right.
\end{equation}
under the assumptions: $\hb, \kb$ are uniformly bounded in
$L^2(\Omega)$, $\lb, \mub \in \R$ are bounded in $\R$, $\o_1, \o_2
\in \R$ are fixed constants. Defining
\begin{equation}\label{eq:alphastar}
\a^*=\begin{cases}
1 & \text{if }N=2\\
1/2 & \text{if }N=3,\\
\end{cases}
\end{equation}
then, by Sobolev embedding, we have that any solution of \eqref{system_main} belongs to $C^{0,\a}$, for every $\a\in (0,\a^*)$ (and even $\a=\a^*$ if $N=3$). As a consequence, in the general case, we can not expect boundedness for every H\"older exponent. In fact we have the following.
\begin{teo}\label{theorem_main}
Let $\ub, \vb$ be solutions of (\ref{system_main}) uniformly bounded in $L^\infty(\Omega)$. Then for every $\a \in (0,\a^*)$ there exists $C>0$, independent of $\b$, such that
\[
\left\|(u_\b,v_\b)\right\|_{C^{0,\a}(\overline{\O})}\leq C\quad\text{for every }\b>0.
\]
\end{teo}

\begin{teo}\label{theorem_uniform_convergence}
Let $\ub, \vb$ be solutions of (\ref{system_main}) uniformly bounded in $L^\infty(\Omega)$. Then there exists $(u,v) \in C^{0,\a}$, $\forall \ \a\in (0,\a^*)$, such that (up to a subsequence) there holds, as $\b\to+\infty$,
\begin{itemize}
\item[(i)] $\ub\to u, \vb\to v$ in $C^{0,\a}(\overline{\Omega})\cap H^1(\Omega)$, $\forall \ \a\in (0,\a^*)$;
\item[(ii)] $u\cdot v\equiv 0$ in $\Omega$ and $\displaystyle \int_\Omega \b u_\b^2 v_\b^2\rightarrow 0$;
\item[(iii)] the limiting functions $u,v$ satisfy the system
 \begin{equation}\label{limiting_eq_for_u_v}\left\{\begin{array}{lll}
-\D u +\l u &=& \o_1 u^3 +h \quad \text{ in } \{u>0\},\\
  -\D v +\mu v &=& \o_2 v^3 +k \quad \text{ in } \{v>0\},
\end{array}\right.
\end{equation}
where $\l:=\lim \l_\b$, $\mu:=\lim \mu_\b$, and $h,k$ denote the $L^2$--weak limits of $h_\b,k_\b$ as $\b\rightarrow +\infty$.
\end{itemize}
\end{teo}
Even though the actual result is stronger (no limitation on $\a$), the proof of Theorem
\ref{theorem_premain} is in fact a particular case of the one of Theorem
\ref{theorem_main}, once one observes that, if $h_\b\equiv
k_\b\equiv0$, then $u_\b$ and $v_\b$, at any fixed $\b$, belong to
$C^{1,\a}$ for every $\a\in(0,1)$. For this reason, we will prove in the details all the
results in the case of system \eqref{system_main}, except the
Lipschitz continuity of the limiting state (section
\ref{sec:lipsch_limit}), that requires $h_\b\equiv
k_\b\equiv0$.

We wish to mention that system \eqref{system_premain} is of great
interest also in the complementary case we do not face, namely when
$\b$ is negative, for instance because of its application to the
study of incoherent solitons in nonlinear optics. For results in
this direction we refer the reader to \cite{ac,dww,mmp,si} and
references therein.

\section{Liouville--type results} \label{section_Liouville}

In this section we prove some nonexistence results in $\R^N$. The
main tools will be the monotonicity formula by Alt, Caffarelli,
Friedman originally stated in \cite{ACF}, as well as some
generalizations made by Conti, Terracini, Verzini
(\cite{ctv3,ctv4}).

\begin{lemma}[(Monotonicity formula)] \label{lemma_monotonicity_formula_disjoint_supports}
Let $u, v \in H^1_{\mathrm{loc}}(\R^N)\cap C(\R^N)$ be nonnegative functions such that $u\cdot v \equiv 0$. Assume moreover that $-\D u\leq 0, -\D v \leq 0$ in $\R^N$ and $u(\xo)=v(\xo)=0$. Then the function
$$
J(r):= \frac{1}{r^2}\int_{B_r(\xo)}\frac{|\nabla u|^2}{\fs}\cdot \frac{1}{r^2}\int_{B_r(\xo)}\frac{|\nabla v|^2}{\fs}
$$
is non decreasing for $r \in (0, +\infty)$.
\end{lemma}

\begin{prop}\label{prop_liouville_disjoint_supports}
Under the same assumptions of Lemma \ref{lemma_monotonicity_formula_disjoint_supports}, assume moreover that for some $\a \in (0,1)$ there holds
\begin{equation}\label{holder_quotient_u_v}
\sup_{x,y \in \R^N}\frac{|u(x)-u(y)|}{|x-y|^\a}, \sup_{x,y \in
\R^N}\frac{|v(x)-v(y)|}{|x-y|^\a} < \infty.
\end{equation}
Then either $u\equiv 0$ or $v\equiv 0$.
\end{prop}

\begin{proof}
In the following we will denote $B_r:=B_r(\xo)$. Assume by
contradiction that neither $u$ nor $v$ is zero, then none of them is
constant since $u(\xo)=v(\xo)=0$. Hence Lemma
\ref{lemma_monotonicity_formula_disjoint_supports} ensures the
existence of a constant $C>0$ such that
\begin{equation}\label{monotonicity_formula_disjoint_supports+holder_first_inequality}
\int_{B_r}\frac{|\nabla u|^2}{\fs}\cdot \int_{B_r}\frac{|\nabla v|^2}{\fs} \geq C r^4
\end{equation}
for every $r$ sufficiently large. Let $\eta_{a,b} \ (0<a<b)$ be any
smooth, radial, cut-off function with the following properties: $0
\leq \eta_{a,b} \leq 1$, $\eta_{a,b}=0$ in $\R^N \setminus B_{b}$,
$\eta_{a,b}=1$ in $B_a$ and $|\nabla \eta_{a,b}| \leq C/(b-a)$.
Given $0<\ep<<r$, let ${A_\ep}:=B_{2r}\setminus B_\ep$ and
$\eta:=\eta_{r,2r}(1-\eta_{\ep,2\ep})$. Testing the inequality
$-\Delta u \leq 0$ with the function $\eta^2 u / \fs$ in the annulus
$A_\ep$, we obtain
\begin{eqnarray*}
\int_{A_\ep} \frac{\eta^2 |\nabla u|^2}{\fs} &\leq&
- \int_{A_\ep} \left[ \frac{2\eta u}{\fs} \nabla u \cdot \nabla\eta +\eta^2 u \nabla u\cdot \nabla\left(\frac{1}{\fs}\right) \right] \\
&\leq& \int_{A_\varepsilon} \left[ \frac{1}{2}\frac{\eta^2|\nabla u|^2}{\fs} + 2 \frac{u^2 |\nabla \eta|^2}{\fs}
-\eta^2 u \nabla u\cdot \nabla\left(\frac{1}{\fs}\right) \right].
\end{eqnarray*}
We can rewrite the last term using the fact that $1/\fs$ is harmonic
in $A_\ep$:
\begin{eqnarray*}
0&=&\int_{A_\ep} \nabla\left(\frac{\eta^2 u^2}{2}\right)\cdot\nabla\left(\frac{1}{\fs}\right)= \\
&=& \int_{A_\ep} \left[ \eta u^2 \nabla\eta\cdot\nabla\left(\frac{1}{\fs}\right) +
\eta^2 u \nabla u\cdot \nabla\left(\frac{1}{\fs}\right) \right],
\end{eqnarray*}
obtaining
\begin{equation*}
\frac{1}{2} \int_{A_\ep} \frac{\eta^2 |\nabla u|^2}{\fs} \leq
\int_{A_\ep}\left[2 \frac{u^2 |\nabla \eta|^2}{\fs}+\eta u^2 \nabla\eta\cdot\nabla\left(\frac{1}{\fs}\right)\right].
\end{equation*}
By the definition of $\eta$, the last expression becomes
\begin{eqnarray*}
\frac{1}{2} \int_{B_r\setminus B_{2\ep}} \frac{|\nabla u|^2}{\fs} \leq
\frac{C}{\ep^N}\int_{B_{2\ep}} u^2 + \frac{C'}{r^2}\int_{B_{2r}\setminus B_r}\frac{u^2}{\fs}
+ \frac{C''}{r}\int_{B_{2r}\setminus B_r}\frac{u^2}{|x-\xo|^{N-1}}.
\end{eqnarray*}
Keeping in mind that $u(x_0)=0$, we let now $\ep\to0$, obtaining
$$
\int_{B_r} \frac{|\nabla u|^2}{\fs} \leq \frac{C'}{r^2}\int_{B_{2r}\setminus B_r}\frac{u^2}{\fs}
+ \frac{C''}{r}\int_{B_{2r}\setminus B_r}\frac{u^2}{|x-\xo|^{N-1}}.
$$
Using the assumptions on $u$, this implies
$$
\int_{B_r} \frac{|\nabla u|^2}{\fs} \leq \frac{C}{r^2} \int_0^{2r} \frac{\rho^{2\a}}{\rho^{N-2}}\rho^{N-1} d\rho +
\frac{C'}{r} \int_0^{2r} \frac{\rho^{2\a}}{\rho^{N-1}}\rho^{N-1} d\rho
\leq C'' r^{2\a}.
$$
Since the same result holds for $v$, we finally obtain
$$
\int_{B_r}\frac{|\nabla u|^2}{\fs} \int_{B_r}\frac{|\nabla v|^2}{\fs} \leq C r^{4\a},
$$
which contradicts
(\ref{monotonicity_formula_disjoint_supports+holder_first_inequality})
for $r$ large and $\a<1$.
\end{proof}

\begin{coro}\label{coro_liouville_harmonic}
Let $u$ be an harmonic function in $\R^N$ such that for some $\a \in (0,1)$ there holds
$$
\sup_{x,y \in \R^N}\frac{|u(x)-u(y)|}{|x-y|^\a} <\infty.
$$
Then $u$ is constant.
\end{coro}

\begin{proof}
If $u\geq0$ or $u\leq 0$, then since $u$ is harmonic it holds that it is a constant (this is the usual nonexistence Liouville result). Otherwise if $u$ changes sign, then we can apply the previous result to its positive and negative parts.
\end{proof}

\begin{rem}\label{rem:not_able}
The previous result does not hold for $\alpha=1$: consider for
instance the function $u(x)=x_1$ (analogously, it is possible to see that also system \eqref{system_limit} below admits non trivial solutions which are globally bounded in Lipschitz norm; these are the main reasons for which
our strategy, as it is, can not apply to prove uniform Lipschitz estimates).
\end{rem}

We shall need a result similar to Proposition \ref{prop_liouville_disjoint_supports}, for functions $u,v$ which do not have disjoint supports, but are positive solutions in $H^1_{\mathrm{loc}}(\R^N)\cap C(\R^N)$ of the system
\begin{equation}\label{system_limit}
\left\{\begin{array}{rlll}
-\D u &=& -u v^2 \quad \text{ in } \R^N\\
-\D v &=& -u^2 v \quad \text{ in } \R^N.
\end{array}\right.
\end{equation}
Again, to obtain a Liouville--type result for the previous system,
we will use a suitable generalization of the monotonicity formula (a
similar idea, even though with slightly different equations, can be
found in \cite{ctv3, ctv4}). To this aim we introduce a $C^1$
auxiliary function
\begin{equation*}
f(r)=\left\{\begin{array}{clll}
\dfrac{2-N}{2}r^2+\dfrac{N}{2} \quad & r\leq1\smallskip\\
\dfrac{1}{r^{N-2}} \quad & r>1
\end{array}\right.
\end{equation*}
and denote $m(|x|):=-\Delta f(|x|)/2$. Notice that $m(|x|)$ is bounded on $\R^N$, vanishes in $\R^N \setminus B_1$ and $m(|x|)\geq 0$ for a.e. $x$.

\begin{lemma}\label{lemma_monotonicity_formula_not_disjoint_supports}
Let $u,v$ be positive solutions of (\ref{system_limit}) and let
$\eps >0$ be fixed. Then there exists $\bar{r} >1$ such that the
function
$$
J(r):= \frac{1}{r^{4-\eps}}\int_{B_r(0)}\left[f(|x|)\left(|\nabla u|^2 + u^2v^2 \right) +m(|x|)u^2\right]
\cdot\int_{B_r(0)}\left[f(|x|)\left(|\nabla v|^2 + u^2v^2 \right)+m(|x|)v^2\right]
$$
is increasing for $r \in (\bar{r}, +\infty)$.
\end{lemma}

\begin{proof}
Let us first evaluate the derivative of $J(r)$ for $r>1$. In order to simplify notations we shall denote $J(r)= J_1(r)J_2(r)/r^{4-\eps}$. Then we have
\begin{equation}\label{J'(r)}
\frac{J'(r)}{J(r)}=-\frac{4-\eps}{r}+\frac{\int_{\partial B_r}f(|x|)(|\nabla u|^2+u^2v^2)}{J_1(r)} +
\frac{\int_{\partial B_r}f(|x|)(|\nabla v|^2+u^2v^2)}{J_2(r)}
\end{equation}
(recall that $m(r)=0$ for $r>1$). We can rewrite the term $J_1$ in a different way: by testing the equation for $u$ with $f(|x|) u$ on $B_r$, we obtain
\begin{eqnarray*}
\int_{B_r}f(|x|)(|\nabla u|^2 + u^2v^2)&=&-\int_{B_r}\nabla(\frac{u^2}{2})\cdot(\nabla f(|x|)) +\int_{\partial B_r} f(|x|) u \parder{u}{\nu}=\\
&=& - \int_{B_r} m(|x|) u^2 + \int_{\partial B_r} \{f(|x|) u \parder{u}{\nu} -\frac{u^2}{2}\parder{}{\nu}(f(|x|))\},
\end{eqnarray*}
which gives
\begin{equation}\label{J_1(r)}
J_1(r)= \frac{1}{r^{N-2}}\int_{\partial B_r}u \parder{u}{\nu} + \frac{N-2}{r^{N-1}}\int_{\partial B_r}\frac{u^2}{2}.
\end{equation}
In order to estimate this quantity we define
\begin{eqnarray*}
\Lambda_1(r):= \frac{r^2 \int_{\partial B_r}(|\nablat u|^2+u^2v^2)}{\int_{\partial B_r}u^2},\qquad
\Lambda_2(r):= \frac{r^2 \int_{\partial B_r}(|\nablat v|^2+u^2v^2)}{\int_{\partial B_r}v^2},
\end{eqnarray*}
where $|\nablat u|^2=|\nabla u|^2-\left| \parder{u}{\nu}\right|^2$. Then for every $\delta \in \R$, by Young's inequality, there holds
\begin{eqnarray*}
\left|\int_{\partial B_r}u \parder{u}{\nu}\right| &\leq& \left(\int_{\partial B_r} u^2\right)^{1/2}
\left(\int_{\partial B_r} (\parder{u}{\nu})^2\right)^{1/2}\\
&\leq& \frac{\sqrt{\Lambda_1(r)}}{2\delta^2 r} \int_{\partial B_r}u^2 + \frac{\delta^2r}{2\sqrt{\Lambda_1(r)}}\int_{\partial B_r}(\parder{u}{\nu})^2\\
&\leq& \frac{1}{2} \left[ \frac{1}{\delta^2}\int_{\partial B_r}\left(|\nablat u|^2+u^2v^2\right)
+\delta^2\int_{\partial B_r}(\parder{u}{\nu})^2\right]\frac{r}{\sqrt{\Lambda_1(r)}}.
\end{eqnarray*}
Substituting in (\ref{J_1(r)}) we obtain
$$
J_1(r) \leq \frac{1}{2r^{N-3}} \left[
\left(\frac{1}{\delta^2\sqrt{\Lambda_1(r)}}+\frac{N-2}{\Lambda_1(r)}\right)\int_{\partial B_r}(|\nablat u|^2 + u^2v^2)+
\frac{\delta^2}{\sqrt{\Lambda_1(r)}}\int_{\partial B_r}(\parder{u}{\nu})^2 \right].
$$
Now we choose $\delta$ in such a way that $\displaystyle \frac{1}{\delta^2\sqrt{\Lambda_1(r)}}+\frac{N-2}{\Lambda_1(r)}=\frac{\delta^2}{\sqrt{\Lambda_1(r)}}$, or equivalently, after some calculation,
$$
\frac{\sqrt{\Lambda_1(r)}}{\delta^2}=\gamma(\Lambda_1(r)),
$$
where $\gamma:\R^+ \to \R$ is defined as
$$
\gamma(x)=\sqrt{\left(\frac{N-2}{2}\right)^2+x}-\frac{N-2}{2}.
$$
We remark that this function plays a crucial role in the proof of the Alt-Caffarelli-Friedman Monotonicity Formula (see \cite{ACF}). Of particular importance is the following property: let $E_1,E_2$ be any couple of disjoint subsets of the sphere  $S^{N-1}$ and denote with $\lambda(E_i)$ the first eigenvalue of the Dirichlet Laplacian on $S^{N-1}$, then
\begin{equation}\label{characteristic_function_inequality}
\gamma(\lambda(E_1))+\gamma(\lambda(E_2)) \geq 2.
\end{equation}
With this choice of $\delta$ we have
$$
J_1(r) \leq \frac{r}{2\gamma(\Lambda_1(r))}\int_{\partial B_r}f(|x|)(|\nabla u|^2+u^2v^2),
$$
(recall that $r>1$ and consequently $f(r)=1/r^{N-2}$) and a similar expression holds also for $J_2$. Substituting in (\ref{J'(r)}) we obtain
$$
\frac{J'(r)}{J(r)}\geq -\frac{4-\eps}{r} + \frac{2\gamma(\Lambda_1(r))}{r} + \frac{2\gamma(\Lambda_2(r))}{r},
$$
therefore it only remains to prove that there exists a $\bar{r}>1$ such that for every $r \geq \bar{r}$ there holds
\begin{equation}\label{characteristic_function_inequality_approximated}
\gamma(\Lambda_1(r))+\gamma(\Lambda_2(r)) > \frac{4-\eps}{2}.
\end{equation}
To this aim we define the functions $\ur(\theta),\vr(\theta):
\partial B_1(0) \to \R$ as $\ur(\theta):=u(r\theta)$,
$\vr(\theta):=v(r\theta)$. Then a change of variables gives
\begin{eqnarray*}
\Lambda_1(r)= \frac{\int_{\partial B_1}(|\nabla \ur|^2+r^2\ur^2\vr^2)}{\int_{\partial B_1}\ur^2},\qquad
\Lambda_2(r)= \frac{\int_{\partial B_1}(|\nabla \vr|^2+r^2\ur^2\vr^2)}{\int_{\partial B_1}\vr^2}.
\end{eqnarray*}
The idea now is to show that the functions $\ur, \vr$ (normalized in
$L^2(\partial B_1)$) converge as $r \to +\infty$ to some functions
having disjoint supports, and then to take advantage of
$(\ref{characteristic_function_inequality})$. Notice first of all
that there exists a constant $C>0$ such that $\int_{\partial
B_1}\ur^2 \geq C$ for $r$ sufficiently large. Indeed assume by
contradiction this is not true, then $\frac{1}{|\partial
B_r|}\int_{\partial B_r}u \to 0$ as $r \to +\infty$, which implies
$u(0)=0$ since $u$ is subharmonic, and this contradicts the
assumption $u>0$. The same result clearly holds also for $\vr$.

Assume (\ref{characteristic_function_inequality_approximated}) does
not hold, then there exists $r_n \to +\infty$ such that
$$
\gamma(\Lambda_1(r_n))+\gamma(\Lambda_2(r_n)) \leq \frac{4-\eps}{2} < 2.
$$
In particular, $\Lambda_1(r_n)$ and $\Lambda_2(r_n)$ are bounded. As a consequence the function
$$
\urtilde:=\frac{\ur}{\| \ur \|_{L^2(\partial B_1)}}
\quad\text{ satisfies }\quad C \geq \Lambda_1(r_n) \geq
\int_{\partial B_1}|\nabla \urtilde |^2
$$
(and an analogous property holds for $\vrtilde:=\vr/\| \vr
\|_{L^2(\partial B_1)}$). This ensures the existence of $\ubar,
\vbar \neq 0$ such that $\urtilde \rightharpoonup \ubar$, $\vrtilde
\rightharpoonup \vbar$ in $H^1(\partial B_1(0))$. Moreover, since
$$
C \geq \Lambda_1(r_n) \geq r_n^2\int_{\partial B_1}\urtilde^2\vrtilde^2
$$
we infer that $\ubar\cdot \vbar\equiv 0$. This immediately provides
$$
\liminf_{n \to +\infty}[\gamma(\Lambda_1(r_n))+\gamma(\Lambda_2(r_n))] \geq \gamma(\lambda(\{\mathrm{supp}(\ubar)\})) + \gamma(\lambda(\{\mathrm{supp}(\vbar)\})),
$$
that is in contradiction with \eqref{characteristic_function_inequality}.
\end{proof}

Now that we have a suitable monotonicity formula we are ready to prove a Liouville--type result for the considered system.

\begin{prop}\label{prop_liouville_not_disjoint_supports}
Let $u,v$ be non negative solutions of (\ref{system_limit}). Assume
moreover that (\ref{holder_quotient_u_v}) holds for some $\a \in
(0,1)$. Then one of the functions is identically zero and the other
is a constant.
\end{prop}

\begin{proof}
We start by noticing that, due to the form of system
\eqref{system_limit}, if one of the functions is 0 or a positive
constant, then the other must be a constant or 0 respectively. Hence
we assume by contradiction that neither $u$ nor $v$ is constant.
Then by the maximum principle $u$ and $v$ are positive, and Lemma
\ref{lemma_monotonicity_formula_not_disjoint_supports} ensures the
existence of a constant $C>0$ such that
\begin{equation}\label{monotonicity_formula_not_disjoint_supports+holder_first_inequality}
\int_{B_r}\left[f(|x|)\left(|\nabla u|^2 + u^2v^2 \right) +m(|x|)u^2\right]
\int_{B_r}\left[f(|x|)\left(|\nabla v|^2 + u^2v^2 \right)+m(|x|)v^2\right] \geq C r^{4-\eps}
\end{equation}
for $r$ sufficiently large. Let $\eta=\eta_{r,2r}$ be the cut-off
function defined in the proof of Proposition
\ref{prop_liouville_disjoint_supports}. By testing the equation for
$u$ with $\eta^2 f u$ on $B_{2r}$ we obtain
\begin{eqnarray*}
\int_{B_{2r}} \eta^2 f\cdot(|\nabla u |^2 +u^2v^2) &=&
-\int_{B_{2r}}\left[2f \eta u \nabla u\cdot\nabla\eta +\eta^2 \nabla \left(\frac{u^2}{2}\right)\cdot\nabla f\right] \leq \nonumber \\
&\leq& \int_{B_{2r}}\left[ \frac{1}{2}f\eta^2 |\nabla u|^2 +2 f u^2 |\nabla\eta|^2 -\nabla\left(\frac{\eta^2 u^2}{2}\right)\cdot \nabla f + u^2 \eta \nabla \eta \cdot \nabla f \right].
\end{eqnarray*}
Recalling that $\D f=-2m$ and testing it with $\eta^2 u^2/2$ in $B_{2r}$ we have
$$
\int_{B_{2r}}\nabla\left(\frac{\eta^2 u^2}{2}\right)\cdot\nabla f = \int_{B_{2r}} \eta^2 u^2 m,
$$
which substituted in 
the previous inequality, together with $m\geq 0$, gives
$$
\int_{B_{2r}} \eta^2 \left[ f\cdot(|\nabla u |^2 +u^2v^2) + m u^2\right] \leq 2 \int_{B_{2r}}\left[2 f u^2 |\nabla\eta|^2 + u^2 \eta \nabla \eta \cdot \nabla f\right].
$$
Now, recalling the definition of $\eta$ and $f$ and using assumption (\ref{holder_quotient_u_v}), we finally obtain
\[
\int_{B_{r}} \left[ f\cdot(|\nabla u |^2 +u^2v^2) + m u^2 \right] \leq
C \int_{B_{2r}\setminus B_r} \frac{u^2}{|x|^N} \leq 
C \int_0^{2r}\frac{\rho^{2\a}}{\rho^N}\rho^{N-1}d\rho \leq
Cr^{2\a},
\]
which contradicts (\ref{monotonicity_formula_not_disjoint_supports+holder_first_inequality}) for $r$ large enough.
\end{proof}

Arguing as above, one can prove the following Liouville--type theorem for systems with an arbitrary number of
densities.

\begin{prop}\label{prop:liouville_k}
Let $k\geq3$ and $u_1,\dots,u_k$ be non negative
solutions of
\begin{equation}
-\D u_i = -u_i\sum_{j\neq i} u_j^2 \quad \text{in } \R^N,
\end{equation}
with the property that, for some $\a\in(0,1)$,
\[
\sup_{x,y \in \R^N}\frac{|u_i(x)-u_i(y)|}{|x-y|^\a} <\infty\quad \text{for every } i.
\]
Then $k-1$ functions are identically zero and the remaining one is
constant.
\end{prop}
\begin{proof}[Sketch of the proof]
We want to see that, for any $i\neq j$, (at least) one between $u_i$ and $u_j$ is identically zero (this, exploiting every possible choice of $i$ and $j$, will readily complete the proof). Assume not, then, by the maximum principle,  $u=u_i$ and $v=u_j$ are positive subsolutions of system \eqref{system_limit}. It is easy to see that Lemma \ref{lemma_monotonicity_formula_not_disjoint_supports} also holds for positive subsolutions of that system; as a consequence, \eqref{monotonicity_formula_not_disjoint_supports+holder_first_inequality} holds for $u=u_i$ and $v=u_j$. But this, reasoning as in the proof of the previous proposition, is in contradiction with the global bound of the H\"older quotients.
\end{proof}

\section{Uniform H\"older continuity}\label{sec:blowup}

This section is mainly devoted to the proof of Theorem
\ref{theorem_main}, that will provide, as a byproduct, also Theorem
\ref{theorem_uniform_convergence}. As we said the strategy we follow
is reasoning by contradiction, in order to perform a blow--up
analysis, and then to use the results of the previous section to
obtain an absurd. To start with, we need the following technical
lemma, which refines the estimate in \cite{ctv4}, Lemma 4.4 (to
which we refer for more details).


\begin{lemma}\label{lemma_exponential_decay}
Let $B_R\subset \R^N$ be any ball of radius $R$. Let $M, A$ be positive constants, $h \in L^2(B_R)$, and let $u\in H^1(B_R)$ be a solution of
$$
\left\{\begin{array}{rlll}
-\Delta u & \leq & -Mu + h & {\rm in }\ B_R\\
u & \geq & 0 & {\rm in}\ B_R\\
u & \leq & A & {\rm on}\ \partial B_R.
\end{array}\right.
$$
Then for every $\ep,\theta>0$ such that $0<\theta<\ep<R$ there holds
$$\|u\|_{L^2(B_{R-\ep})}\leq \frac{2AR}{\ep-\theta} e^{-\theta \sqrt{M}}+\frac{1}{M}\|h\|_{L^2(B_R)},$$
where $B_{R-\ep}$ is the ball of radius $R-\ep$ which shares its center with $B_R$.
\end{lemma}

\begin{proof} We can estimate $u$ as $|u|\leq |u_1|+|u_2|$, where $u_1,u_2$ are defined by
$$
\left\{\begin{array}{rlll}
-\Delta u_1 & = & -Mu_1 & {\rm in }\ B_R\\
u_1 & \geq & 0 & {\rm in}\ B_R\\
u_1 & = & u & {\rm on}\ \partial B_R
\end{array}\right.
\qquad
\left\{\begin{array}{rlll}
-\Delta u_2 & = & -Mu_2 + h & {\rm in }\ B_R\\
u_2 & = & 0 & {\rm on}\ \partial B_R.
\end{array}\right.
$$
In order to estimate $u_1$ we shall make use of Lemma 4.4 in
\cite{ctv4}, where it is proved that
\begin{equation*}
\left\{\begin{array}{rl}
& \psi_1^{''}(r)+\frac{N-1}{r}\psi_1'(r) = M\psi_1(r) \\
& \psi_1(R)=A>0 \\
& \psi_1'(0)=0
\end{array}\right.
\Rightarrow
\left\{\begin{array}{rll}
&\psi_1(0)>0, \quad \psi_1'(r)>0 \quad & r \in [0,+\infty) \\
&\psi_1(r) \leq \psi_1(0) e^{r\sqrt{M}} \quad & r \in [0,+\infty) \\
&\psi_1(r) \geq \frac{\psi_1(0)\bar{r}}{2e^{\bar{r}\sqrt{M}} } \frac{e^{r\sqrt{M}}}{r} \quad & r \in [\bar{r},+\infty).
\end{array}\right.
\end{equation*}
Choosing $\bar{r}=\ep -\theta$, with $\theta \in (0,\ep)$, we obtain the following inequalities
$$
\begin{array}{rll}
&\psi_1(R-\ep)\leq \psi_1(0) e^{(R-\ep)\sqrt{M}} \\
&A=\psi_1(R) \geq \frac{\psi_1(0)(\ep-\theta)}{2e^{(\ep-\theta)\sqrt{M}}}\frac{e^{R\sqrt{M}}}{R},
\end{array}
$$
which imply
$$
\psi_1(r) \leq \psi_1(R-\ep) \leq \frac{2 A R}{\ep-\theta} e^{-\theta \sqrt{M}}, \qquad \text{ for all}\ r \in [0,R-\ep].
$$
Defining $v_1(x)=\psi_1(|x-x_0|)$ (where $x_0$ is the center of $B_R$) and using the maximum principle we infer $0\leq u_1(x)\leq v_1(x)$. To obtain an upper estimate for $u_2$, let us now multiply the equation for $u_2$ by $u_2$ itself and integrate; having zero boundary conditions we have
$$
\int_{B_R}M u_2^2 \leq \int_{B_R}|\nabla u_2|^2+Mu_2^2 = \int_{B_R}h u_2 \leq \left(\int_{B_R}\frac{h^2}{M}\right)^{1/2}
\left(\int_{B_R} M u_2^2 \right)^{1/2},
$$
and therefore $\|u_2\|_{L^2(B_R)} \leq \frac{1}{M}\|h\|_{L^2(B_R)}$. In conclusion we have $\|u\|_{L^2(B_{R-\ep})} \leq \|u_1\|_{L^2(B_{R-\ep})} + \|u_2\|_{L^2(B_R)}$ which gives the desired estimates.
\end{proof}

\subsection{Normalization and blow--up}\label{subsec:blowup}

To start with, we recall the standard non--uniform regularity properties for solutions to system \eqref{system_main}.
\begin{rem}\label{rem:H^2}
Let $u_\b,v_\b$ be solutions of \eqref{system_main}. Then, since $h_\b,k_\b$ belong to $L^2(\O)$, and $\O$ is bounded and regular, by elliptic regularity theory it holds
\[
u_\b,v_\b\in H^2(\O)\quad \text{that implies} \quad u_\b,v_\b\in C^{0,\a}(\overline{\O})
\]
for every $\a\in(0,\a^*)$, where $\a^*$ is defined as in \eqref{eq:alphastar}. Let us mention that, if $h_\b\equiv k_\b\equiv0$, then, by a bootstrap argument, we can choose $\a^*=1$ also in dimension $N=3$.
\end{rem}

Coming to the proof of Theorem \ref{theorem_main}, let us assume by
contradiction that, for some $\alpha \in (0,\a^*)$, up to a subsequence it holds
$$
L_\b:=\max\left\{ \max_{x,y\in\overline{\Omega}} \frac{|u_\b(x)-u_\b(y)|}{|x-y|^\a},\
\max_{x,y\in\overline{\Omega}} \frac{|v_\b(x)-v_\b(y)|}{|x-y|^\a}\right\} \longrightarrow +\infty
$$
as $\b \rightarrow +\infty$. We can assume that $L_\b$ is achieved,
say, by $u_\b$ at the pair $(x_\b,y_\b)$. We observe that
\[
|x_\b-y_\b|\rightarrow 0\quad\text{ as }\quad\b\rightarrow +\infty,
\]
since we have
$|x_\b-y_\b|^\a=|u_\b(x_\b)-u_\b(y_\b)|/L_\b \leq 2
\|u_\b\|_\infty/L_\b\leq 2C/L_\b\rightarrow 0$.

The idea now is to consider an uniformly $\a$-H\"older continuous
blow-up with center at $x_\b$. Keeping this in mind, let us define
the rescaled functions
$$
\ubar_\b(x)=\frac{1}{L_\b r_\b^\a}u_\b(x_\b+r_\b x), \quad  \vbar_\b(x)=\frac{1}{L_\b r_\b^\a}v_\b(x_\b+r_\b x),
\qquad \ \text{for }x\in \Omega_\b:=\frac{\Omega-x_\b}{r_\b},
$$
where $r_\b \rightarrow 0$ will be chosen later. Depending on the
asymptotic behavior of the distance $d(x_\b,\partial \Omega)$ and
on $r_\b$, we have $\Omega_\b\rightarrow \Omega_\infty$, where
$\Omega_\infty$ is either $\R^N$ or an half-space (when
$d(x_\b,\partial \Omega)/r_\b\rightarrow \infty$ or the limit is
finite, respectively).

First of all we observe that the $\ubar_\b, \vbar_\b$'s are uniformly $\a$-H\"older continuous for every choice of $r_\b$, with H\"older constant equal to one:
\begin{equation} \label{holderquotient_for_ubar, vbar}
\max\left\{ \max_{x,y\in\Omegabar_\b} \frac{|\ubar_\b(x)-\ubar_\b(y)|}{|x-y|^\a},\ \max_{x,y\in\Omegabar_\b} \frac{|\vbar_\b(x)-\vbar_\b(y)|}{|x-y|^\a}\right\}
=\frac{\left|\ubar_\b(0)-\ubar_\b\left(\frac{y_\b-x_\b}{r_\b}\right)\right|}{\left|\frac{y_\b-x_\b}{r_\b}\right|^\a}=1.
\end{equation}
Moreover the rescaled functions satisfy the following system in $\Omega_\b$:
\begin{equation}\label{system_in_Omega_k}
\left\{\begin{array}{lll}
-\D \ubar_\b + \lb r_\b^2 \ubar_\b &=& \omega_1 M_\b \ubar_\b^3 -\b M_\b \ubar_\b \vbar_\b^2 + \hbar_\b(x)  \\
-\D \vbar_\b + \mub r_\b^2 \vbar_\b &=& \omega_2 M_\b \vbar_\b^3 - \b M_\b \ubar_\b^2 \vbar_\b + \kbar_\b(x) \smallskip\\
\ubar_\b,\ \vbar_\b \in H^1_0(\Omega_\b),
    \end{array}
\right.
\end{equation} where
\[
M_\b:= L_\b^2 r_\b^{2\a+2},
\]
and
$$
\hbar_\b(x):=\frac{r_\b^{2-\a}}{L_\b}h_\b(x_\b+r_\b x), \quad \kbar_\b(x):=\frac{r_\b^{2-\a}}{L_\b}k_\b(x_\b+r_\b x).
$$

\begin{rem}\label{remark_L^2_estimates_rescaled_functions}
Since $u_\b,v_\b$ are $L^{\infty}(\Omega)$--bounded, $h_\b, k_\b$
are $L^2(\Omega)$--bounded, $\lambda_\b,\mu_\b$ are bounded in $\R$,
and $r_\b\rightarrow 0$, $L_\b\rightarrow +\infty$, by direct
calculations it is easy to see that
\[
\begin{split}
&\lambda_\b r_\b^2 \ubar_\b,\ \mu_\b r_\b^2 \vbar_\b \to0 \quad\text{in }L^\infty(\Omega_\b)\\
&\omega_1 M_\b \ubar_\b^3,\ \omega_2 M_\b \vbar_\b^3 \to 0 \quad\text{in }L^\infty(\Omega_\b)\\
&\hbar_\b,\ \kbar_\b\to 0 \quad\text{in }L^2(\Omega_\b).
\end{split}
\]
\end{rem}

In order to manage the different parts of the proof, we will need to
make different choices of the sequence $r_\b$. Once $r_\b$ is
chosen, we wish to pass to the limit (on compact sets), and to this
aim we will use Ascoli--Arzel\`{a}'s Theorem. Now, since the
$\ubar_\b,\vbar_\b$'s are uniformly $\a$-H\"older continuous, it
suffices to show that $\{ \ubar_\b(0)\}$, $\{ \vbar_\b(0)\}$ are
bounded in $\b$.  The following lemma provides a sufficient
condition on $r_\b$ for such a bound to hold.

\begin{lemma} \label{lemma_wise_r_k}
Under the previous notations, let $r_\b \to 0$ as $\b \to +\infty$ be such that
\begin{itemize}
\item[(i)] $\dfrac{|y_\b-x_\b|}{r_\b} \leq R'$ for some $R'>0$,
\item[(ii)] $\b M_\b \nrightarrow 0$.
\end{itemize}
Then $\{\ubar_\b(0)\}, \{\vbar_\b(0)\}$ are uniformly bounded in $\b$.\\
\end{lemma}

\begin{proof}
Assume by contradiction that $\{\ubar_\b(0)\}$ is unbounded, and let
$R\geq R'$. Since the $\ubar_\b$'s are uniformly H\"older continuous
and vanish on $\partial \Omega_\b$, we can consider $\b$
sufficiently large such that $B_{2R}(0)\subset \Omega_\b$. Moreover
since $\b M_\b \nrightarrow 0$, we have that
$$I_\b:= \inf_{B_{2R}(0)} \b M_\b \ubar_\b \longrightarrow +\infty.$$
Now since $\b M_\b \ubar_\b \leq \b M_\b \ubar_\b^2$ in $B_{2R}(0)$
and (similarly to Remark
\ref{remark_L^2_estimates_rescaled_functions}) $\|\omega_2 M_\b
\vbar_\b^2\|_{L^\infty(B_{2R})}\to 0$ as $\b\to+\infty$, we have
\begin{eqnarray*}
 -\Delta \vbar_\b &=& -\mu_\b r_\b^2 \vbar_\b + \omega_2 M_\b \vbar_\b^3 - \b M_\b \ubar_\b^2 \vbar_\b + \kbar_\b\\
&\leq & -\frac{I_\b}{2} \vbar_\b + \kbar_\b.
\end{eqnarray*}

In order to use Lemma \ref{lemma_exponential_decay}, we need to show that $\vbar_\b$ is bounded on $\partial B_{2R}(0)$. With this in mind, let us choose a cut-off function $\eta$ that vanishes outside $B_{2R}(0)$. Then by testing the second equation in \eqref{system_in_Omega_k} with $\eta^2 \vbar_\b$ in $B_{2R}(0)$, we obtain
\begin{equation*}
\int_{B_{2R}(0)}\{\eta^2 |\nabla \vbar_\b|^2 + 2\eta \vbar_\b \nabla \eta \cdot \nabla \vbar_\b + \mu_\b r_\b^2\vbar_\b^2\eta^2\} \leq \int_{B_{2R}(0)}\{\omega_2 M_\b \vbar_\b^4 \eta^2 - I_\b \eta^2 \vbar_\b^2 + \kbar_\b \eta^2 \vbar_\b\},
\end{equation*}
and thus
\begin{eqnarray*}
\int_{B_{2R}(0)}\{\frac{1}{2}\eta^2|\nabla \vbar_\b|^2+I_\b \eta^2 \vbar_\b^2\}&\leq& 2\int_{B_{2R}(0)} \{|\nabla \eta|^2\vbar_\b^2 + |\mu_\b|r_\b^2 v_\b^2\eta^2 + |\omega_2| M_\b \vbar_\b^4 \eta^2 + \kbar_\b \eta^2\vbar_\b\}\\
&\leq& C(R) (\sup_{B_{2R}(0)}\vbar_\b^2+1).
\end{eqnarray*}
On the other hand, since $\vbar_\b$ is uniformly H\"older continuous,
$$I_\b \int_{B_{2R}(0)}\eta^2 \vbar_\b^2 \geq I_\b C'(R) \inf_{B_{2R}(0)} \vbar_\b^2 \geq I_\b C'(R) \sup_{B_{2R}(0)} \vbar_\b^2 - I_\b C''(R).$$
Therefore, putting together the two previous inequalities, we obtain
$$I_\b C(R) \sup_{B_{2R}(0)}\vbar_\b^2\leq C'(R) (\sup_{B_{2R}(0)}\vbar_\b^2+1) + I_\b C''(R)$$ which implies the boundedness of $\vbar_\b$ in $B_{2R}(0)$ (in particular on $\partial B_{2R}(0)$).

Thus we can apply Lemma \ref{lemma_exponential_decay}, which gives
$$\|\vbar_\b\|_{L^2(B_R)}\leq C e^{-C'\sqrt{I_\b}}+ \frac{2}{I_\b}\|\kbar_\b\|_{L^2(B_{2R})}.$$
Hence
$$\|\b M_\b \ubar_\b \vbar_\b\|_{L^2(B_R)}\leq (I_\b + \b M_\b(4R)^\a)\|\vbar_\b\|_{L^2(B_r)}\leq
2I_\b
(Ce^{-C'\sqrt{I_\b}}+\frac{2}{I_\b}\|\kbar_\b\|_{L^2(B_{2R})})\rightarrow 0$$ when $\b\rightarrow +\infty$. This, together with Remark \ref{remark_L^2_estimates_rescaled_functions} and the boundedness of $\vbar_\b$, gives
\begin{equation}\label{delta_ubar_delta_vbar_go_to_0_in_L2}\|\Delta \ubar_\b\|_{L^2(B_R)}\rightarrow 0\end{equation}for every $R\geq R'$.

Consider now $\tilde{u}_\b(x):= \ubar_\b(x)-\ubar_\b(0)$. By the
uniform H\"older continuity and Ascoli--Arzel\`a's Theorem we know
that $\tilde{u}_\b\rightarrow \tilde{u}_\infty$ on compact sets.
Moreover by \eqref{delta_ubar_delta_vbar_go_to_0_in_L2} we have that
$\ubar_\b$ is bounded in $C^{0,\gamma}_{\rm loc}$, with $\gamma\in
(0,\a^*)$ (in fact, Theorem 8.12 of \cite{GT} gives us boundedness
in $W^{2,2}$, and the result follows by Sobolev imbbedings). As a
consequence we obtain:
\begin{equation} \label{holder_equal_1_in_lemma_wise_r_k}
\max_{x,y\in \overline{\Omega}_\infty} \frac{|\tilde{u}_\infty(x)-\tilde{u}_\infty(y)|}{|x-y|^\a} =1.
\end{equation}
Indeed notice that by assumption $(i)$, ${(y_\b-x_\b)}/{r_\b}$ must
converge up to a subsequence. But it can not be
${(y_\b-x_\b)}/{r_\b}\rightarrow 0$, otherwise we would have
(considering an $\varepsilon>0$ sufficiently small)
\begin{equation}\label{eq:tende_a_0}
\frac{\left|\ubar_\b(0)-\ubar_\b
\left(\frac{y_\b-x_\b}{r_\b}\right)\right|}{\left|\frac{y_\b-x_\b}{r_\b}\right|^\a}=\left|\frac{y_\b-x_\b}{r_\b}\right|^\varepsilon
\frac{\left|\tilde{u}_\b
(0)-\tilde{u}_\b\left(\frac{y_\b-x_\b}{r_\b}\right)\right|}{\left|\frac{y_\b-x_\b}{r_\b}\right|^{\a+\varepsilon}}\leq
C \left|\frac{y_\b-x_\b}{r_\b}\right|^\varepsilon \to 0
\end{equation}
with $\b$, which contradicts \eqref{holderquotient_for_ubar, vbar}.
Therefore, there is an $a\in \R^N\setminus \{ 0 \} $ such that
$({y_\b-x_\b})/{r_\b}\rightarrow a$, and hence the left hand side of
\eqref{holderquotient_for_ubar, vbar} also passes to the limit in
$\b$, providing \eqref{holder_equal_1_in_lemma_wise_r_k}.

Finally, we have that $\Delta \tilde{u}_\infty=0$ in $\Omega_\infty$. Now if $\Omega_\infty=\R^N$, by Corollary \ref{coro_liouville_harmonic} $\tilde{u}_\infty$ is a constant, in contradiction with \eqref{holder_equal_1_in_lemma_wise_r_k}. On the other hand, if $\Omega_\infty$ is an half space, we have that $\tilde{u}_\infty=0$ on $\partial \Omega_\infty$ and thus we can extend it by symmetry as an harmonic function in the whole $\R^N$, obtaining the same contradiction.

We have shown that $\{\ubar_\b(0)\}$ is bounded. Let us now check that the same happens with $\{\vbar_\b(0)\}$. In order to do so, we have to make some small changes to the previous argument. Assume then that $\{\vbar_\b(0)\}$ is unbounded, and consider the quantity (for $R\geq R'$ fixed)
$$\tilde{I}_\b:= \inf_{B_{2R}(0)}\b M_\b \vbar_\b^2 \rightarrow +\infty.$$
We have
$$-\Delta \ubar_\b\leq -\frac{\tilde{I}_\b}{2}\ubar_\b+\hbar_\b$$ and $\ubar_\b$ is bounded on $\partial B_{2R}(0)$. Therefore by Lemma \ref{lemma_exponential_decay}
$$\|\ubar_\b\|_{L^2(B_r)}\leq C e^{-C'\sqrt{\tilde{I}_\b}}+\frac{2}{\tilde{I}_\b}\|\hbar_\b\|_{L^2(B_{2R})}$$ and hence
$$\|\b M_\b \ubar_\b \vbar_\b^2\|_{L^2(B_R)}\rightarrow 0$$ as $\b\rightarrow +\infty$. Once again this gives $\|\Delta \ubar_\b\|_{L^2(B_R)}\rightarrow 0$ and the proof follows as before.
\end{proof}
Using the previous lemma we can now quantify the asymptotic relation between $\b$, $L_\b$ and $|x_\b-y_\b|$.
\begin{lemma}\label{lem:conv0}
Under the previous notation, we have (up to a subsequence)
\[
\b L_\b ^2 |x_\b-y_\b|^{{2\a + 2}}\to+\infty.
\]
\end{lemma}
\begin{proof}
By contradiction assume that  $\b L_\b ^2 |x_\b-y_\b|^{{2\a + 2}}$ is bounded. Then we can choose
\[
r_\b=(\b L_\b^2)^{-\frac{1}{2\a +2}}\qquad\text{(and thus $\b M_\b=1$)},
\]
in such a way that the assumptions of Lemma \ref{lemma_wise_r_k} are
satisfied and thus $\{\ubar_\b(0)\}, \{\vbar_\b(0)\}$ are bounded.
By uniform H\"older continuity and Ascoli--Arzel\`a's theorem we
have that, up to a subsequence, there exist $u_\infty,v_\infty$ such
that $\ubar_\b \rightarrow u_\infty,\ \vbar_\b \rightarrow v_\infty$
uniformly in the compact subsets of $\overline{\Omega}_\infty$.
Since $\b M_\b=1$ and by Remark
\ref{remark_L^2_estimates_rescaled_functions}, we have that $\Delta
\ubar_\b,\ \Delta \vbar_\b$ are bounded in $L^2_{\mathrm{loc}}$
and therefore the same happens to $\ubar_\b,\vbar_\b$ in
$C^{0,\gamma}_\mathrm{loc}(\overline{\Omega}_\infty)$, for all $\gamma \in (0,\a^*)$.
We are now going to show that, as a consequence, $u_\infty,v_\infty$
are $\a$-H\"older continuous and that the maximum of the H\"older
quotients is given by:
\begin{equation} \label{Db_bounded_holder_uinfty}
\max_{x,y\in \overline{\Omega}_\infty} \frac{|u_\infty(x)-u_\infty(y)|}{|x-y|^\a} =1.
\end{equation}
Indeed notice that we cannot have $ ({y_\b-x_\b})/{r_\b}\rightarrow
0$, otherwise we would obtain the same contradiction as in
\eqref{eq:tende_a_0}. Therefore, there is an $a\in\R^2\backslash
\{0\}$ such that $({y_\b-x_\b})/{r_\b}\rightarrow
a$, and hence the left hand side of \eqref{holderquotient_for_ubar,
vbar} also passes to the limit in $\b$, providing
\eqref{Db_bounded_holder_uinfty}. Moreover, at the limit we have
\begin{equation} \nonumber
\left\{\begin{array}{llll}
-\Delta u_\infty &=& -u_\infty v_\infty^2 \quad &\text{in} \ \Omega_\infty\\
-\Delta v_\infty &=& -u_\infty^2 v_\infty \quad &\text{in} \ \Omega_\infty.
\end{array}\right.
\end{equation}

If $\Omega_\infty=\R^N$, then by Proposition
\ref{prop_liouville_not_disjoint_supports} $u_\infty,v_\infty$ are
constants, which contradicts \eqref{Db_bounded_holder_uinfty}.

On the other hand, let $\Omega_\infty$ be equal to an half-space. Since $u_\infty=0$ on $\partial \Omega_\infty$, we can extend it to the whole space by even symmetry and obtain a function satisfying the hypotheses of Proposition \ref{prop_liouville_disjoint_supports} (apply it to the pair $({u_\infty}_{|_{\Omega_\infty}},{u_\infty}_{|_{\R^N\setminus \overline{\Omega}_\infty}})$ -- were we consider both functions extended by 0 -- and choose for $x_0$ any point of $\partial \Omega_\infty$). Therefore $u_\infty \equiv 0$, which contradicts (\ref{Db_bounded_holder_uinfty}).
\end{proof}

Now we are in a position to define our choice of $r_\b$ and to deduce the convergence of the blow--up sequences.

\begin{lemma}\label{lemma_H^1_convergence}
Let
\[
r_\b=|x_\b-y_\b|.
\]
Then there exist $u_\infty,v_\infty\in C^{0,\alpha}(\R^N)$ such that, as $\b\to+\infty$ (up to subsequences), there holds
\begin{itemize}
\item[(i)] $\ubar_\b\to u_\infty$, $\vbar_\b\to v_\infty$, uniformly in compact subsets of $\Omega_\infty=\R^N$; moreover
\item[(ii)] for any fixed $r>0$ and $x_0\in \R^N$ there holds
$\dis\int_{B_r(x_0)} \b M_\b \ubar_\b^2\vbar_\b^2 \to 0$; as a consequence
\item[(iii)] $\| \ubar_\b - u_\infty \|_{H^1(B_r(x_0))} \to 0$, $\|\vbar_\b-v_\infty\|_{H^1(B_r(x_0))} \to 0$.
\end{itemize}
\end{lemma}

\begin{proof}
With this choice of $r_\b$, we obtain $\b M_\b=\b L_\b ^2 |x_\b-y_\b|^{{2\a + 2}}\to+\infty$ by Lemma \ref{lem:conv0}. Once again the assumptions of Lemma \ref{lemma_wise_r_k} are satisfied and hence, reasoning as in the initial part of the proof of Lemma \ref{lem:conv0}, we deduce that the rescaled functions $\ubar_\b,\vbar_\b$ converge uniformly to some
$u_\infty,v_\infty$, in every compact set of $\Omegabar_\infty$.
In this situation \eqref{holderquotient_for_ubar, vbar} writes
$$
1=\max \left\{  \max_{x,y\in\Omegabar_\b} \frac{|\ubar_\b(x)-\ubar_\b(y)|}{|x-y|^\a},
\max_{x,y\in\Omegabar_\b} \frac{|\vbar_\b(x)-\vbar_\b(y)|}{|x-y|^\a}\right\}
=  \left|\ubar_\b(0)-\ubar_\b\left(\frac{y_\b-x_\b}{r_\b}\right)\right|
$$
and hence by $L_{\mathrm{loc}}^\infty(\Omegabar_\b)$ convergence,
$u_\infty,v_\infty$ are globally $\alpha$--H\"older continuous and in
particular
\begin{equation}\label{pre_Db_unbounded_holder_uinfty}
\max_{x\in \partial B_1(0)\cap \Omegabar_\infty}|u_\infty(0)-u_\infty(x)|=1.
\end{equation}
Now if $\Omega_\infty$ is an half-space we can proceed exactly as in the last part of the proof of Lemma \ref{lem:conv0}, obtaining a contradiction. Therefore $\Omega_\infty=\R^N$, and $(i)$ is proved.

In order to prove the second part of the lemma, let us fix any ball
$B_r(x_0)$ of $\R^N$, and let $\b$ be large so that
$B_r(x_0)\subset\O_\b$. Let us consider a smooth cut-off function $0
\leq \eta \leq 1$ such that $\eta=1$ in $B_r$, $\eta=0$ in
$\R^N \setminus B_{2r}$. Testing the equation for $\ubar_\b$ 
with $\eta$, we obtain (since the $\ubar_\b$'s are uniformly bounded in $B_{2r}$)
\begin{equation}\label{int_Br_b_M_b_ubar_vbar^2_bounded}
\int_{B_r} \b M_\b \ubar_\b \vbar_\b^2 \leq
\int_{B_{2r}} |\ubar_\b \Delta \eta- \lambda_\b r_\b^2 \eta \ubar_\b + \omega_1 M_\b \eta \ubar_\b^3  + \eta \hbar_\b |\leq C
\end{equation}
and analogously $\displaystyle \int_{B_r} \b M_\b \ubar_\b^2 \vbar_\b \leq C$.
This immediately implies that
\begin{equation}\label{eq:pre_disjoint}
u_\infty\cdot v_\infty \equiv 0\quad\text{ in }\R^N,
\end{equation}
providing
\begin{eqnarray}\label{eq:hugo}
\int_{B_r} \b M_\b \ubar_\b^2 \vbar_\b^2 &\leq& \|\ubar_\b\|_{L^\infty(B_r\cap \{u_\infty=0\})}\int_{B_r} \b M_\b \ubar_\b \vbar_\b^2 + \|\vbar_\b\|_{L^\infty(B_r)\cap \{v_\infty=0\}} \int_{B_r}\b M_\b \ubar_\b^2 \vbar \nonumber\\
&\leq & C\left(\|\ubar_\b\|_{L^\infty(B_r\cap \{u_\infty=0\})} + \|\vbar_\b\|_{L^\infty(B_r\cap \{v_\infty=0\})} \right)\rightarrow 0,
\end{eqnarray}
which is $(ii)$.

Finally, integrating the equation for $\ubar_\b$ in $B_r$, we have
\begin{equation}\label{normal_derivative_estimates}
\left|\int_{\partial B_r} \parder{\ubar_\b}{\nu}\,d\sigma\right| \leq
\int_{B_r} \b M_\b \ubar_\b \vbar_\b^2 + \int_{B_r}|\lambda_\b r_\b^2 \ubar_\b -\omega_1 M_\b \ubar_\b^3 - \hbar_\b|\leq C,
\end{equation}
which also gives, testing the equation for $\ubar_\b$ with $\ubar_\b$ itself, $\int_{B_r} |\nabla \ubar_\b|^2\leq C$.
Doing the same with $\vbar_\b$, we obtain the weak $H^1$--convergence $\ubar_\b\rightharpoonup u_\infty$,
$\vbar_\b \rightharpoonup v_\infty$
. Finally by testing the equation for $\ubar_\b$ with $\ubar_\b-u_\infty$ we obtain
\[
\begin{split}
\int_{B_r}\nabla \ubar_\b\cdot \nabla (\ubar_\b-u_\infty)\leq
\|\ubar_\b-u_\infty\|_{L^\infty(B_r)}\cdot&\left(\int_{\partial B_r}\left|\partial_\nu \ubar_\b\right|+\right. \smallskip\\
&\left.+\int_{B_r}|-\lambda_\b r_\b^2 \ubar_\b + \omega_1 M_\b \ubar_\b^3 - \b M_\b \ubar_\b \vbar_\b^2 + \hbar_\b|\right)
\end{split}
\]
and therefore we proved $(iii)$ by uniform convergence and estimates \eqref{normal_derivative_estimates}, \eqref{int_Br_b_M_b_ubar_vbar^2_bounded} (the convergence of $\vbar_\b$ is analogous).
\end{proof}

In the following lemma we collect the properties enjoyed by the limiting states $u_\infty$, $v_\infty$.

\begin{lemma}\label{lem:properties_u_infty}
Let $u_\infty$, $v_\infty$ be defined as in Lemma \ref{lemma_H^1_convergence}. Then the following holds.
\begin{itemize}
\item[(i)] $u_\infty\cdot v_\infty \equiv 0$ in $\R^N$;
\item[(ii)] $\dis\max_{x\in \partial B_1(0)}|u_\infty(0)-u_\infty(x)|=1$ (in particular, $u_\infty$ is not constant);
\item[(iii)] $ \left\{\begin{array}{lll}
  -\D u_\infty &=& 0  \quad \text{ in } \{u_\infty>0\},\\
  -\D v_\infty &=& 0  \quad \text{ in } \{v_\infty>0\}.
\end{array}\right.
$
\end{itemize}
\end{lemma}

\begin{proof}
Properties $(i)$ and $(ii)$ are simply \eqref{eq:pre_disjoint} and
\eqref{pre_Db_unbounded_holder_uinfty}, respectively.
Let us check that $u_\infty$ is harmonic in the (open) set $\{x\in \R^N:\ u_\infty(x)>0\}$ (the same is true for
$v_\infty$ in the set $\{x\in\R^N:\ v_\infty(x)>0\}$). Given any point $x_0$ such that $u_\infty(x_0)>0$, we have to
find a neighborhood of it where $u_\infty$ is harmonic. By continuity we can consider a ball $B_\delta(x_0)$ where
$u_\infty\geq 2 \gamma>0$, and hence by locally $L^\infty$ convergence $\ubar_\b\geq \gamma>0$ in $B_\delta(x_0)$ for large $\b$.
Therefore we have
$$-\Delta \vbar_\b \leq -\b M_\b \frac{\gamma^2}{2}\vbar_\b+\kbar_\b$$
and thus, using Lemma \ref{lemma_exponential_decay}, we obtain
$$\|\vbar_\b\|_{L^2(B_{\delta/2})}\leq C e^{-C' \sqrt{\b M_\b}}+\frac{1}{\b M_\b}\|\kbar_\b\|_{L^2(B_{\delta/2})}.$$
Hence
$$\|\b M_\b \ubar_\b \vbar_\b^2\|_{L^2(B_{\delta/2})}\rightarrow 0$$and, using also Remark \ref{remark_L^2_estimates_rescaled_functions}, we conclude that $\|\Delta \ubar_\b\|_{L^2(B_{\delta/2})}\rightarrow 0$, which implies the harmonicity of $u_\infty$ in $B_{\delta/2}(x_0)$.
\end{proof}

\begin{rem}\label{rem:v_infty=0}
By the previous lemmas we obtain that $u_\infty$ must vanish somewhere in $\R^N$ (indeed if not $u_\infty$ would be a positive non--constant harmonic function in $\R^N$, a contradiction), and also $v_\infty$ must vanish somewhere (otherwise we would have $u_\infty\equiv0$ in $\R^N$, again a contradiction). This, by continuity, implies that $u_\infty$ and $v_\infty$ must have a common zero, thus they satisfy all the assumptions of Proposition \ref{prop_liouville_disjoint_supports}. Since $u_\infty$ is not constant, we deduce that
\[
v_\infty\equiv0\qquad\text{in }\R^N.
\]
Moreover, we have
\[
\{x:\,u_\infty(x)=0\}\neq\emptyset,\quad\text{ and }\quad \{x:\,u_\infty(x)>0\}\text{ is connected.}
\]
This last claim is due to the fact that, was $\{u_\infty>0\}$ non trivially decomposed into
$\O_1\cup\O_2$, then again $u=u_\infty|_{\O_1}$ and $v=u_\infty|_{\O_2}$ would be non--zero and satisfy  the assumptions of Proposition \ref{prop_liouville_disjoint_supports}, a contradiction.
\end{rem}

\subsection{Almgren's Formula}\label{subsec:almgren}

In order to conclude the proof of Theorem \ref{theorem_main} we will show that $u_\infty$ is radially homogeneous; this crucial information will come from a generalization of the Almgren's Monotonicity Formula. This formula was first introduced in \cite{Alm} and used for instance in \cite{caflin,GL} to prove some regularity issues related to free boundary problems. The aim is to study the monotonicity properties of the functions
$$
E(r)=\frac{1}{r^{N-2}}\int_{B_r}\left\{|\nabla u_\infty|^2+|\nabla v_\infty|^2\right\}, \qquad H(r)= \frac{1}{r^{N-1}}\int_{\partial B_r} \left\{u_\infty^2+v_\infty^2\right\},
$$
and of the Almgren's quotient (where it is defined)
$$
N(r)=\frac{E(r)}{H(r)},
$$
where $u_\infty$, $v_\infty$ are defined in Lemma \ref{lemma_H^1_convergence} and $B_r$ is centered at a fixed $x_0$ (with respect to the literature, our definition of $H$ involves the averages of the densities, not of their oscillations). It is worthwhile noticing that the result we prove for $u_\infty$, $v_\infty$ in fact holds for any non trivial, strong $H^1_{\mathrm{loc}}$--limits of variational systems; indeed, we will perform the proof without using all the other properties we collected about $u_\infty$, $v_\infty$. The reason for this is that we will need a similar result, for different functions, in Section \ref{sec:lipsch_limit}.

\begin{prop}\label{prop_almgren_rescaled}
Under the above notations, for every $x_0\in\R^N$ there exists $r_0\geq 0$ such that, for every $r>r_0$, $H(r)\neq 0$, and
\[
N(r)\text{ is an absolutely continuous, non decreasing function}
\]
such that
\begin{equation}\label{derivative_of_H(r)}
\frac{d}{dr} \log (H(r))=\frac{2N(r)}{r}.
\end{equation}
Moreover if $N(r)\equiv \gamma$ for all $r>r_0$, then $r_0=0$ and $u_\infty(x)= r^\gamma g_1(\theta)$,
$v_\infty(x)= r^\gamma g_2(\theta)$ in $\R^N$, for some functions $g_1,g_2$ (where $(r,\theta)$ denote the polar coordinates).
\end{prop}

\begin{proof}
Up to a translation, we can suppose $B_r=B_r(0)$. We divide the proof into steps.

\textbf{Approximated quotients.} Let $0<r_1<r_2$ be such that $H(r)\neq 0$ in $[r_1,r_2]$ (they exist for sure, since $u_\infty\not\equiv 0$ and it is continuous). Let us check that the conclusions of the proposition follow in this interval (the existence of $r_0$ as claimed will be obtained only later). To evaluate derivatives of $E(r)$, $H(r)$ and $N(r)$ we have to face two main problems: first, it is not clear how regular these functions are; second, we have no global equation for $u_\infty, v_\infty$. To overcome these difficulties, the idea is to consider analogous functions, that will result to be $C^1$, for the approximated problem (\ref{system_in_Omega_k}), and then to pass to the limit as $\beta \to +\infty$. In order to simplify notations we will denote for the moment $u:=\ubar_\b$ and the same for $\vbar_\b, \hbar_\b, \kbar_\b$. We then define the approximated Almgren's quotient
$$
N_\beta(r)=\frac{E_\b(r)}{H_\b(r)},
$$
where
$$ E_\b(r)=\frac{1}{r^{N-2}}\int_{B_r}\left\{ |\nabla u|^2+|\nabla v|^2 + r_\b^2 (\lambda_\b u^2 + \mu_\b v^2)-M_\b (\omega_1 u^4+\omega_2 v^4) + 2 \b M_\b u^2 v^2\right\},
$$
$$
H_\b(r)=\frac{1}{r^{N-1}}\int_{\partial B_r}\left\{u^2+v^2\right\}.
$$
We also observe that, by multiplying system \eqref{system_in_Omega_k} by $(u,v)$ and integrating in $B_r$, we obtain
\begin{equation}\label{E_b_new_formulation}
E_\b(r)=\frac{1}{r^{N-2}}\int_{\partial B_r} \left\{u\,{\partial_\nu}u+v\,{\partial_\nu}v\right\} +
\frac{1}{r^{N-2}}\int_{B_r} \left\{h(x)u+k(x)v \right\}
\end{equation}
(the boundary integrals above, and all the following ones, are well
defined, for $\b$ fixed and for \emph{every} $r$, by Remark
\ref{rem:H^2} and by the continuous immersion of $H^2(B_r)$ into
$H^1(\partial B_r)$).

\textbf{Derivatives of $E_\b$, $H_\b$.}
In order to compute the derivatives of these expressions, we consider the rescaled function $u_r(x):=u(rx)$ and similar expressions for $v,h,k$. System \eqref{system_in_Omega_k} now becomes
\begin{equation}\label{system_for_ur_vr}
\left\{\begin{array}{lll}
-\Delta u_r + r^2 \lambda_\b r_\b^2 u_r   &=& r^2 \omega_1 M_\b u_r^3 - r^2 \b M_\b u_r v_r^2 + r^2 h_r \\
-\Delta v_r + r^2 \mu_\b r_\b^2 v_r       &=& r^2 \omega_2 M_\b v_r^3 - r^2 \b M_\b u_r^2 v_r + r^2 k_r.
\end{array}
\right.
\end{equation}
Performing a change of variables $x=ry$ in $E_\b(r)$ we obtain
$$
E_\b(r)=\int_{B_1}\left\{ |\nabla u_r|^2+|\nabla v_r|^2 +r^2 r_\b^2 (\lambda_\b u_r^2+\mu_\b v_r^2)-r^2 M_\b(\omega_1 u_r^4 + \omega_2 v_r^4)+ 2 r^2 \b M_\b u_r^2 v_r^2\right\},
$$
and hence (Remark \ref{rem:H^2} implies that $E_\b$ is in fact $C^1$ in $r$)
$$
\begin{array}{lll}
E_\b'(r) &=& \displaystyle 2 \int_{B_1} \left\{\nabla u_r \cdot \nabla (\nabla u(rx)\cdot x))+\nabla v_r\cdot \nabla(\nabla v(rx)\cdot x)\right\}+\\
&& \displaystyle + 2r^2\int_{B_1} (r_\b^2 \lambda_\b u_r- 2 M_\b \omega_1 u_r^3 + 2 \b M_\b u_r v_r^2)(\nabla u(rx)\cdot x)+\\
&&\displaystyle + 2r^2 \int_{B_1} (r_\b^2 \mu_\b v_r- 2 M_\b \omega_2 v_r^3 + 2 \b M_\b u_r^2 v_r)(\nabla v(rx)\cdot x)+\\
&&\displaystyle +2r\int_{B_1} \left\{r_\b^2(\lambda_\b u_r^2+\mu_\b v_r^2)- M_\b (\omega_1 u_r^4+\omega_2 v_r^4)+ 2 \b M_\b u_r^2 v_r^2\right\}
.\end{array}
$$
Multiplying the  first equation in \eqref{system_for_ur_vr} by $(\nabla u(r x)\cdot x) $, the second one by $(\nabla v(r x)\cdot x)$, integrating by parts in $B_1$ and substituting the result in the previous expression, it follows:
$$
\begin{array}{lll}
E_\b'(r) &=& \displaystyle \frac{2}{r^{N-2}}\int_{\partial B_r}\left\{\left({\partial_\nu}u\right)^2+\left({\partial_\nu}v\right)^2\right\} + \frac{2}{r^{N-1}}\int_{B_r}\left\{h(x)(\nabla u\cdot x) + k(x)(\nabla v\cdot x)\right\}+\\
&&\displaystyle + \frac{2}{r^{N-1}}\int_{B_r}\left\{r_\b^2(\lambda_\b u^2+\mu_\b v^2)- M_\b (\omega_1 u^4 + \omega_2 v^4) + 2 \b M_\b u^2 v^2 \right\}-\\
&&\displaystyle -\frac{2}{r^{N-1}}\int_{B_r}\left\{ \omega_1 M_\b u^3 (\nabla u \cdot x) + \omega_2 M_\b v^3 (\nabla v \cdot x) \right\}+\\
&&\displaystyle + \frac{2}{r^{N-1}}\int_{B_r}\left\{\b M_\b u (\nabla u\cdot x) v^2 + \b M_\b u^2 v (\nabla v\cdot x)\right\}.
\end{array}
$$
Using the divergence theorem, we can rewrite some terms:
\begin{multline*}
\frac{2}{r^{N-1}}\int_{B_r}\left\{\b M_\b u (\nabla u \cdot x) v^2  + \b M_\b u^2 v (\nabla v \cdot x)\right\}= \frac{\b M_\b}{r^{N-1}}\int_{B_r} \nabla (u^2 v^2)\cdot x=\\
= -\frac{N}{r^{N-1}}  \int_{B_r} \b M_\b u^2 v^2 + \frac{1}{r^{N-2}} \int_{\partial B_r}\b M_\b u^2 v^2;
\end{multline*}
\begin{multline*}
-\frac{2}{r^{N-1}}\int_{B_r}\left\{\omega_1 M_\b u^3(\nabla u\cdot x)+\omega_2 M_\b (\nabla(v)\cdot x)\right\} =\\= -\frac{2}{r^{N-1}}\int_{B_r}\left\{\omega_1 M_\b (\nabla(u^4)\cdot x)+\omega_2 M_\b (\nabla(v^4)\cdot x)\right\}=\\ =\frac{N}{2r^{N-1}}\int_{B_r}\left\{\omega_1 M_\b u^4+\omega_2 M_\b v^4\right\}-\frac{1}{2r^{N-2}}\int_{\partial B_r}\left\{\omega_1 M_\b u^4 + \omega_2 M_\b v^4\right\};
\end{multline*}
obtaining at the end
\[
\begin{split}
E_\b'(r) &=  \frac{2}{r^{N-2}}\int_{\partial B_r}\left\{\left({\partial_\nu}u\right)^2+\left({\partial_\nu}v\right)^2\right\} + \frac{2}{r^{N-1}}\int_{B_r}\left\{h(x)(\nabla u\cdot x) + k(x)(\nabla v\cdot x)\right\}+\\
&
+ \frac{1}{r^{N-1}}\int_{B_r}\left\{2r_\b^2(\lambda_\b u^2+\mu_\b v^2)+\left(\frac{N}{2}-2\right)M_\b(\omega_1 u^4+\omega_2 v^4)+(4-N) \b M_\b u^2 v^2\right\}+\\
&
+\frac{1}{2 r^{N-2}}\int_{\partial B_r}\left\{2\b M_\b u^2 v^2 - \omega_1 M_\b u^4-\omega_2 M_\b v^4\right\}.
\end{split}
\]
Using the same ideas, we also obtain, for the $C^1$ function $H_\b$,
$$
H'_\b(r)=\frac{2}{r^{N-1}}\int_{\partial B_r} \left\{u\,{\partial_\nu}u + v \,{\partial_\nu}v\right\}.
$$

\textbf{Estimate of $N_\b(r+\delta)-N_\b(r)$.} At this point, let us recover the original notations $\ubar_\b, \vbar_\b, \hbar_\b, \kbar_\b$. Recalling equation \eqref{E_b_new_formulation} we can compute $N'_\b$ in $(r_1,r_2)$ as
\begin{multline*}
N_\b'(r) 
= \frac{2}{r^{2N-3} H_\b^2(r)}\left\{ \int_{\partial B_r} \left[\left(\partial_\nu \ubar_\b\right)^2 + \left({\partial_\nu \vbar_\b}\right)^2 \right]\cdot \int_{\partial B_r}(\ubar_\b^2+\vbar_\b^2)\right.-\\
-\left.
\left[\int_{\partial B_r} \left(\ubar_\b {\partial_\nu \ubar_\b}+\vbar_\b {\partial_\nu \vbar_\b}\right)\right]^2\right\} + R_\b(r),
\end{multline*}
where
\[
\begin{split}
R_\b(r) &=
\frac{2}{r^{N-1} H_\b(r)} \int_{B_r} \left\{\hbar_\b(x) (\nabla \ubar_\b\cdot x) +\kbar_\b(x) (\nabla \vbar_\b\cdot x)\right\}+\\
&
\displaystyle + \frac{1}{r^{N-1}H_\b(r)}\int_{B_r}\left\{2 r_\b^2(\lambda_\b \ubar_\b^2+\mu_\b \vbar_\b^2)+\frac{N-4}{2}M_\b (\omega_1 \ubar_\b^4 + \omega_2 \vbar_\b^4)+(4-N)\b M_\b \ubar_\b^2 \vbar_\b^2 \right\}+ \\
&
+\frac{1}{2r^{N-2}H_\b(r)}\int_{\partial B_r}\left\{2\b M_\b \ubar_\b^2 \vbar_\b^2-\omega_1 M_\b \ubar_\b^4-\omega_2 M_\b^4\right\}-\\
&
-\frac{2}{r^{2N-3}H_\b^2(r)}\int_{\partial B_r}\left\{\ubar_\b\, {\partial_\nu} \ubar_\b +\vbar_\b\,
{\partial_\nu} \vbar_\b\right\}\int_{B_r}\left\{ \hbar_\b(x) \ubar_\b+ \kbar_\b(x) \vbar_\b\right\}.
\end{split}
\]
Notice that, since $H_\b(r) \neq 0$, for every $\delta>0$ such that $r,r+\delta \in (r_1,r_2)$, there exists a constant $C>0$ depending only on $r_1, r_2$ and $\delta$ such that
\begin{eqnarray*}
\int_r^{r+\delta} |R_\b(s)|ds &\leq&
C \int_{B_{r_2}}\{|\hbar_\b| |\nabla \ubar_\b| + |\kbar_\b| |\nabla \vbar_\b| + r_\b^2 (\ubar_\b^2+\vbar_\b^2)+ M_\b (\ubar_\b^4 + \vbar_\b^4) + \\
&& + \b M_\b \ubar_\b^2 \vbar_\b^2 + |\omega_1|M_\b \ubar_\b^3 |\nabla \ubar_\b| + |\omega_2| M_\b \vbar_\b^3 |\nabla \vbar_\b|\}+\\
& &+\|M_\b \ubar_\b^4+M_\b\vbar_\b^4\|_{L^\infty(B_r)}  + C  \int_{B_{r_2}}\left\{ |\hbar_\b| |\ubar_\b| + |\kbar_\b| |\vbar_\b|\right\}  \longrightarrow 0
\end{eqnarray*}
as $\beta\rightarrow +\infty$, where we used Remark \ref{remark_L^2_estimates_rescaled_functions}, Lemma \ref{lemma_H^1_convergence}, $(iii)$ and \eqref{normal_derivative_estimates}.
Therefore,
\begin{eqnarray}
N_\b(r+\delta)-N_\b(r) &=& \int_r^{r+\delta}\frac{2}{s^{2N-3} H_\b^2(s)}\left\{ \int_{\partial B_s} \left[\left(\partial_\nu \ubar_\b\right)^2 + \left({\partial_\nu \vbar_\b}\right)^2 \right]\cdot \int_{\partial B_s}(\ubar_\b^2+\vbar_\b^2)\right.-\nonumber\\
&&-\left.
\left[\int_{\partial B_s} \left(\ubar_\b {\partial_\nu \ubar_\b}+\vbar_\b {\partial_\nu \vbar_\b}\right)\right]^2\right\}+o_\b(1).\nonumber
\end{eqnarray}

\textbf{Derivatives of $N$, $H$, $E$, $\log H$.} Now we are in a position to pass to the limit in $\b$. Indeed, Lemma \ref{lemma_H^1_convergence}, $(iii)$ (that is, strong convergence) ensures that $N_\b(r)\rightarrow N(r)$ for every $r$. Moreover it implies the existence of a function $f(\rho)\in L^1(r_1,r_2)$ such that, up to a subsequence, $\int_{\partial B_\rho}|\nabla \ubar_\b|^2\leq f(\rho)$ and $\int_{\partial B_\rho}|\nabla \ubar_\b|^2\rightarrow \int_{\partial B_\rho} |\nabla u_\infty|^2$ a.e. for $\rho \in (r_1,r_2)$ (and analogously for $\vbar_\b$). Hence, letting $\b\rightarrow +\infty$ in the previous equation we readily obtain that $N$ is absolutely continuous and that (for almost every $r$)
\begin{eqnarray}\label{N'(r)=...}
N'(r) &=& \frac{2}{r^{2N-3} H^2(r)}\left\{\int_{\partial B_r} \left[\left(\partial_\nu u_\infty\right)^2 + \left(\partial_\nu v_\infty\right)^2\right] \cdot \int_{\partial B_r}(u_\infty^2 +v_\infty^2) - \right. \nonumber \\ && - \left.\left[\int_{\partial B_r}\left( u_\infty \partial_\nu u_\infty + v_\infty \partial_\nu v_\infty\right)  \right]^2   \right\}\geq 0,
\end{eqnarray}
by H\"older inequality. This implies that $N(r)$ is increasing in $[r_1,r_2]$ and in addition gives an explicit expression for the derivative. Reasoning as above, we can conclude that
\begin{eqnarray*}
H'(r) &=& 
\frac{2}{r^{N-1}}\int_{\partial B_r} \left\{ u_\infty\,{\partial_\nu} u_\infty + v_\infty {\partial_\nu} v_\infty\right\},\\
H(r)&=&\lim_{\b\rightarrow +\infty} H_\b(r) = \frac{1}{r^{N-1}}\int_{\partial B_r} \left\{u_\infty^2 + v_\infty^2\right\},\\
E(r)&=&\lim_{\b\rightarrow +\infty} E_\b(r)= \frac{1}{r^{N-2}}\int_{\partial B_r} \left\{ u_\infty\,{\partial_\nu} u_\infty + v_\infty {\partial_\nu} v_\infty\right\},
\end{eqnarray*}
(where we used \eqref{E_b_new_formulation} to obtain the last limit) and therefore a direct calculation gives \eqref{derivative_of_H(r)} for $r\in (r_1,r_2)$. Incidentally, we observe that equation \eqref{derivative_of_H(r)} implies that $\log H$, and hence $H$, are $C^1$--functions.

\textbf{Existence of $r_0$.} Equality \eqref{derivative_of_H(r)} also implies that when $H(r)>0$, then $H'(r)\geq 0$ and therefore there exists $r_0:=\inf\left\{r>0:\ H(r)\neq 0\right\}$ such that $H(r)\neq 0$ for every $r>r_0$. Hence everything we have done so far is true in $(r_0,+\infty)$.

\textbf{Case $N(r)$ constant.} Let us now analyze what happens when $N(r)\equiv \gamma$ for all $r>r_0$. By \eqref{derivative_of_H(r)} we have
$$
\frac{d}{dr} \log(H(r))=\frac{2\gamma}{r} = \frac{d}{dr} \log (r^{2\gamma})
$$
for all $r>r_0$. By considering $\bar{r}>r>r_0$ and integrating the previous equality between $r$ and $\bar{r}$ we have
$$
H(r)=H(\bar{r})\left(\frac{r}{\bar{r}}\right)^{2\gamma}.
$$
Now, if $H(r_0)=0$, since $H$ is continuous, the previous equation implies that $r_0=0$; on the other hand, if $H(r_0)\neq0$, then $r_0=0$  by definition. Moreover, by \eqref{N'(r)=...},
if $N(r)$ is constant then there exists $C(r)$ such that $({\partial_\nu u_\infty},{\partial_\nu v_\infty})=C(r) (u_\infty,v_\infty)$, which gives $u_\infty=f(r)g_1(\theta)$, $v_\infty=f(r)g_2(\theta)$, with $f(r)>0$ for $r>0$. Now since $u_\infty$ is harmonic in $\left\{u_\infty>0\right\}$ and $v_\infty$ is harmonic in $\left\{v_\infty>0\right\}$, we finally infer $u_\infty(x)=r^\gamma g_1(\theta)$, $v_\infty(x)= r^\gamma g_2(\theta)$ in
$\R^N$.
\end{proof}

\begin{rem}\label{rem:almgren_k}
Starting from system \eqref{system_premain_k}, one can perform the same blow--up argument than above, obtaining in particular that, for the limiting states $(u_{1,\infty},\dots,u_{k,\infty})$, a result analogous to Proposition \ref{prop_almgren_rescaled} holds, with the choice
\[
E(r)=\frac{1}{r^{N-2}}\int_{B_r}\sum_{i=1}^k|\nabla u_{i,\infty}|^2, \qquad H(r)= \frac{1}{r^{N-1}}\int_{\partial B_r} \sum_{i=1}^k u_{i,\infty}^2.
\]
\end{rem}

\subsection{Proof of the main results.}

\begin{proof}[End of the proof of Theorem \ref{theorem_main}]
By Lemma \ref{lemma_H^1_convergence} we know that the blow--up limiting profiles $u_\infty$ and $v_\infty$ are globally $\alpha$--H\"older continuous. Moreover, by Remark \ref{rem:v_infty=0}, $v_\infty\equiv0$ and we can choose 
\[
x_0\text{ such that }u(x_0)=0.
\]
If $N(r)$ is defined as in Proposition \ref{prop_almgren_rescaled}, then we claim that $N(r)\equiv \alpha$ for all $r>r_0$. Indeed, according to that proposition, assume that there exists $\bar{r}>r_0$ such that $N(\bar{r})\leq \alpha-\varepsilon$. Then by monotonicity, for all $r_0<r<\bar{r}$ we have  $N(r)\leq \alpha-\varepsilon$ and
$$
\frac{d}{dr}\log(H(r))\leq \frac{2}{r}(\alpha-\varepsilon),
$$
hence (integrating between $r$ and $\bar{r}$) we have $Cr^{2\alpha-2\varepsilon}\leq H(r)$ for all $r_0<r<\bar{r}$. On the other hand, by the $\a$--H\"older continuity and the fact that $u_\infty(x_0)=v_\infty(x_0)=0$, we also have $H(r)\leq C'(r-r_0)^{2\alpha}$, a contradiction. On the other hand, if $N(\bar{r})\geq \alpha+\varepsilon$, then by monotonicity $N(r)\geq \alpha+\varepsilon$ for all $r>\bar{r}$, and thus
$$
\frac{d}{dr}\log H(r)\geq \frac{2}{r}(\alpha+\varepsilon),
$$
which implies (integrating between $\bar{r}$ and $r$  and again by the $\a$--H\"older continuity) that $C r^{2\alpha+2\varepsilon}\leq H(r)\leq C' r^{2\alpha}$ for large $r$, a contradiction.

Therefore $N(r)\equiv \alpha$ for all $r>r_0$, and by the previous proposition we know that $r_0=0$ and $u_\infty(x)=r^\alpha g_1(\theta)$. This implies that the null set $\G=\{u_\infty=0\}$ is a cone with respect to $x_0$. Since this can be done for any $x_0\in\G$, we obtain that $\G$ is in fact a cone with respect to each of its points, and thus it is a linear subspace of $\R^N$. Moreover, again by Remark \ref{rem:v_infty=0}, $\G$ has dimension strictly smaller than $N-1$, otherwise $\{u_\infty>0\}$ would be disconnected. But then
$u_\infty$ turns out to be a non--negative, non--constant function in $H^1_{\mathrm{loc}}(\R^N)$, which is harmonic on the complement of a set of zero (local) capacity. That is, it is harmonic on the whole $\R^N$, a contradiction.
\end{proof}
\begin{proof}[Proof of Theorem \ref{theorem_uniform_convergence}] By Theorem \ref{theorem_main}, for every $\a<\a'<\a^*$ there exists a constant $C>0$ such that $\|(\ub,v_\b)\|_{C^{0,\a'}} \leq C$, for every $\b>0$. 
By compact embedding, we obtain, up to a subsequence, the existence of $(u,v)\in C^{0,\a'}$ that are strong $C^{0,\a}$--limits of $(u_\b,v_\b)$. By uniqueness of the limit, this proves that
\[
(u_\b,v_\b) \to (u,v) \quad\text{in }C^{0,\a}(\Omegabar)\text{ for every }\a<\a^*.
\]
To obtain the other claims of the theorem, we reason as in the proof of Lemma \ref{lemma_H^1_convergence}.
Testing system \eqref{system_main} with $(u_\b,v_\b)$ we obtain
\[
\int_{\O} |\nabla\ub|^2 \leq \int_\O \left(-\lb \ub + \o_1 \ub^3 +\hb\right)u_\b,
\]
and an analogous inequality for $\vb$. By uniform convergence, the right hand side is bounded and then $(u_\b,v_\b)$ is bounded in $H^1_0$. Thus, again up to a subsequence, we have
\[
(u_\b,v_\b) \rightharpoonup (u,v) \quad\text{weakly in }H^1_0(\O).
\]
On the other hand, integrating system \eqref{system_main} we have
\[
-\int_{\partial\O} \partial_\nu\ub +\b\int_\O  \ub \vb^2= \int_\O \left(-\lb \ub + \o_1 \ub^3 +\hb\right).
\]
Again, the right hand side is bounded and, by Hopf lemma, $\partial_\nu\ub<0$ on $\partial\O$. We infer
\[
\b\int_\O  \ub \vb^2\leq C,\qquad\b\int_\O  \ub^2 \vb\leq C
\]
not depending on $\b$. This immediately provides $u\cdot v\equiv 0$ almost everywhere in $\O$, and, in turn, reasoning as \eqref{eq:hugo},
\[
\b\int_\O  \ub^2 \vb^2\to 0\quad\text{as }\b\to+\infty,
\]
that completes the proof of $(ii)$. Now we can test system \eqref{system_main} with $(u_\b-u,v_\b-v)$, obtaining
\[
\int_{\O} \nabla\ub\cdot\nabla(u_\b-u) \leq \|u_\b-u\|_{L^\infty}\int_\O \left(-\lb \ub + \o_1 \ub^3 - \ub \vb^2 +\hb\right),
\]
and the same for $v$. By uniform convergence we infer convergence in norm, and hence strong $H^1_0$--convergence of $(u_\b,v_\b)$ to $(u,v)$, and also $(i)$ is proved. Finally, to prove $(iii)$, we observe that, by continuity of the limiting profile, we know that $\{u>0\}$ is an open set. Therefore, given $\xo\in\{u>0\}$, there exists $B_\d(\xo)$ such that $u\geq2\g>0$ in $B_\d(\xo)$, for some positive constant $\g$. Let us show that the equation is satisfied in this open neighborhood. By $(i)$ there holds $\ub \geq \g$ in $B_\d(\xo)$ for large $\b$, therefore
$$
\int_{B_\d(\xo)}\b\ub\vb^2 \leq \frac{1}{\g}\int_{B_\d(\xo)}\b\ub^2\vb^2 \to 0,
$$
because of $(ii)$. By testing the equation with a test function $\phi \in C_0^1(B_\d(\xo))$ we obtain
$$
\int_{B_\d(\xo)}(\nabla\ub\cdot\nabla\phi + \lb\ub\phi)=\int_{B_\d(\xo)}(\o_1\ub^3-\b\ub\vb^2+\hb)\phi,
$$
and the previous estimate together with the $H^1$--convergence conclude the proof.
\end{proof}
\begin{proof}[Proof of Theorems \ref{theorem_premain} and \ref{theorem_lipschitz}]
As we just noticed, with one small change in the previous arguments
one can prove also these two theorems, except for the Lipschitz
continuity of the limiting profile $(u,v)$, which will be the object
of the following section. In dimension $N=2$, since $\a^*=1$, then
the theorems follow directly from Theorems \ref{theorem_main} and
\ref{theorem_uniform_convergence}. In dimension $N=3$, according to
Remark \ref{rem:H^2}, if $h_\b\equiv k_\b\equiv0$ then we can choose
$\a^*=1$ and repeat, as they are, all the arguments in this section. Then Theorem
\ref{theorem_premain} straightly follows, while the proof of Theorem \ref{theorem_lipschitz}
will be completed by Proposition \ref{prop:lip} and Remark \ref{rem:lip_boundry_2} below.
\end{proof}
\begin{rem}\label{rem:k}
With exactly the same strategy it is also possible to prove
analogous results for $L^\infty$--bounded, positive solutions of
system \eqref{system_premain_k}. The only differences are pointed
out in Proposition \ref{prop:liouville_k} and in Remark
\ref{rem:almgren_k}.
\end{rem}

\section{Lipschitz continuity of the limiting
profile}\label{sec:lipsch_limit}

Throughout all this section, let $(u,v)\in C^{0,\a}\cap H^1_0$ denote the limiting profile
introduced in Theorem \ref{theorem_uniform_convergence}, and $h_\b,k_\b\equiv 0$ (that is, we are dealing with system
\eqref{system_premain}). As we noticed, in this case the uniform
H\"older continuity result holds for every $\a \in (0,1)$ also if
$N=3$. In such a situation, although we are not able to prove uniform Lipschitz
continuity of the solutions with respect to $\b$ (see Remark \ref{rem:not_able}), one can prove that
the limiting profile is in fact Lipschitz continuous. To be more
precise, we will first give the details of the proof of the local
Lipschitz continuity of $(u,v)$, and then we will advise (in Remarks
\ref{rem:lip_boundry_1} and \ref{rem:lip_boundry_2}) how this proof
can be modified in order to obtain the Lipschitz regularity up to
the boundary of $\O$. Also, after Remarks \ref{rem:almgren_k} and
\ref{rem:k}, the reader will easily see how this result holds true for $k$--tuples of densities that are solutions of system \eqref{system_premain_k}.

Let us fix a (regular) domain $\Omegatilde\subset\subset \Omega$, and let us
define the null set
\[
\Gamma=\{x\in\Omegatilde: u(x)=v(x)=0\}\neq\Omegatilde
\]
(from now on, we will exclude the trivial case $(u,v)\equiv(0,0)$ in $\Omegatilde$,
which obviously enjoys Lipschitz continuity).
\begin{prop}\label{prop:lip}
Let $(u,v)$ be the limiting profile introduced in Theorem
\ref{theorem_uniform_convergence}, $h_\b\equiv k_\b\equiv0$ and $\Omegatilde\subset\subset \Omega$.
Then $(u,v)\in W^{1,\infty}(\Omegatilde)$.
\end{prop}
Again, in order to prove the proposition, the main tool will be
the Almgren's Monotonicity Formula introduced in Section
\ref{subsec:almgren}, with some small change in its definition. In
fact, due to the fact that the limiting profiles satisfy system
\eqref{eq:limit_sys}, the natural definitions for $E(r)$ and $H(r)$
are
\begin{eqnarray*}
E(r)=E_{x_0}(r)&=&\frac{1}{r^{N-2}} \int_{B_r(\xo)} \left( |\nabla u|^2 + |\nabla v|^2 +\l u^2+\mu v^2-\o_1 u^4 -\o_2 v^4\right),\\
H(r)=H_{x_0}(r)&=&\frac{1}{r^{N-1}} \int_{\partial B_r(\xo)}(u^2+v^2),
\end{eqnarray*}
where, here and in the following,
\[
x_0\in \overline{\Gamma}\quad\text{ and }\quad r<\rbar_1:=\mathrm{dist}(\Omegatilde,\partial\O).
\]
In this setting, we have that $E(r)/H(r)$ is no longer necessarily
positive. To overcome this fact, we define a modified Almgren's
quotient as
\begin{equation*}
N(r)= \frac{E(r)}{H(r)}+1=\frac{E(r)+H(r)}{H(r)}.
\end{equation*}
With this choice, $N$ turns out to be non negative (where it is
defined).

\begin{lemma}\label{lem:minoration_for_E+H}
There exists $\rbar_2<\rbar_1$ such that for every $0<r\leq\rbar_2$ and for every $x_0$ we have
\begin{equation*}
E(r)+H(r) \geq \frac{1}{2}\left[\frac{1}{r^{N-2}} \int_{B_r(x_0)}(|\nabla u|^2+|\nabla v|^2)+
\frac{1}{r^{N-1}}\int_{\partial B_r(x_0)}(u^2 +v^2) \right]\geq 0.
\end{equation*}
\end{lemma}
\begin{proof}
We shall make use of the following formulation of Poincar\'e's
inequality: for every $w\in H^1_{\mathrm{loc}}(\R^N)$, every $x_0$ and every $r>0$ there
holds
$$
\frac{1}{r^N}\int_{B_r}w^2 \leq \frac{1}{N-1}\left[ \frac{1}{r^{N-2}}\int_{B_r}|\nabla w|^2+\frac{1}{r^{N-1}}\int_{\partial B_r}w^2 \right].
$$
Recalling that $u,v\in L^\infty(\Omega)$, let now $C>0$ be such that
$$
\left|\frac{1}{r^{N}} \int_{B_r} \l u^2+\mu v^2-\o_1 u^4 -\o_2 v^4\right| \leq \frac{C}{r^N}\int_{B_r}(u^2+v^2)
$$
(hence $C$ depends on $u$, $v$, $\o_i$, $\l$, $\mu$, but not on $x_0$ and $r$). Then
$$
E(r)+H(r)\geq \frac{1}{r^{N-2}} \int_{B_r}(|\nabla u|^2+|\nabla v|^2)+\frac{1}{r^{N-1}}\int_{\partial B_r}(u^2 +v^2)
-r^2\frac{C}{r^N}\int_{B_r}(u^2+v^2),
$$
and Poincar\'e's inequality immediately implies that, for $r\leq\rbar_2$ sufficiently small (independent of the choice of $\xo$), the lemma holds.
\end{proof}

Now, with the new notations of this section, let us present a result which corresponds to Proposition \ref{prop_almgren_rescaled} in this context.

\begin{prop}\label{prop_almgren_interior}
There exist $\rbar\leq\rbar_2$ and $C>0$ such that, for every $x_0\in \overline{\Gamma}$ and $0<r\leq\rbar$, we have $H(r)\neq 0$,
$$
N'(r)\geq -2C r N(r),\qquad\text{and thus }\Ntilde(r):=e^{Cr^2}N(r)\text{ is non decreasing.}
$$
Moreover
\begin{equation}\label{derivative_H(r)_Ntilde}
\frac{d}{dr} \log (H(r))=\frac{2}{r}(N(r)-1).
\end{equation}
\end{prop}

\begin{rem}\label{rem_Gamma_has_empty_interior}
During the proof of this proposition we will also see that $\Gamma$ has empty interior.
\end{rem}

\begin{proof}
We will follow closely the proof of Proposition \ref{prop_almgren_rescaled}.

\textbf{Proof when $H(r)\neq0$.} Let us first suppose that there is an interval $[r_1,r_2]$, with $r_2<\rbar_2$ (defined in the previous lemma), such that $H(r)>0$ in $[r_1,r_2]$. Again, we first consider the approximated problem
\begin{eqnarray*}
\Eb(r)&=&\frac{1}{r^{N-2}} \int_{B_r} \left( |\nabla\ub|^2 + |\nabla\vb|^2 +\l_\b \ub^2+\mu_\b \vb^2-\o_1 \ub^4 -\o_2 \vb^4
+2\b \ub^2\vb^2\right),\\
\Hb(r)&=&\frac{1}{r^{N-1}} \int_{\partial B_r}(\ub^2+\vb^2),\\
\Nb(r)&=&\frac{\Eb(r)+\Hb(r)}{\Hb(r)}.
\end{eqnarray*}
Proceeding exactly as in Proposition \ref{prop_almgren_rescaled} we obtain
\begin{eqnarray*}
\Hb'(r)&=&\frac{2}{r^{N-1}}\int_{\partial B_r}\{\ub\parder{\ub}{\nu}+\vb\parder{\vb}{\nu}\}\ =\ \frac{2}{r}\Eb(r),\\
\Eb'(r)&=&\frac{2}{r^{N-2}}\int_{\partial B_r}\{\left(\parder{\ub}{\nu}\right)^2+\left(\parder{\vb}{\nu}\right)^2\} +\Rb(r),
\end{eqnarray*}
where
\begin{eqnarray*}
\Rb(r)&=&\frac{1}{r^{N-1}}\int_{B_r}\left\{2\l_\b\ub^2+\mu_\b\vb^2+
\frac{N-4}{2}(\o_1\ub^4+\o_2\vb^4)+(4-N)\b\ub^2\vb^2\right\}-  \\
&&-\frac{1}{2r^{N-2}}\int_{\partial B_r}\{\o_1\ub^4+\o_2\vb^4-2\b\ub^2\vb^2\}.
\end{eqnarray*}
Hence in $(r_1,r_2)$ there holds, by H\"older inequality,
\[
\frac{\Nb'(r)}{\Nb(r)}=\frac{\Eb'(r)\Hb(r)-\Eb(r)\Hb'(r)}{(\Eb(r)+\Hb(r))\Hb(r)}\geq\frac{\Rb(r)}{\Eb(r)+\Hb(r)},
\]
(recall that, by Lemma \ref{lem:minoration_for_E+H}, $H(r)>0$ implies $E(r)+H(r)> 0$, and $N(r)>0$). Now we can let $\b\to+\infty$, 
obtaining
$$
\frac{N'(r)}{N(r)}\geq \frac{R(r)}{E(r)+H(r)},
$$
with
\[
R(r)=\frac{1}{r^{N-1}}\int_{B_r}\left\{2\l u^2+2\mu v^2+\frac{N-4}{2}(\o_1 u^4+\o_2 v^4) \right\}-  
\frac{1}{2r^{N-2}}\int_{\partial B_r}\{\o_1 u^4+\o_2 v^4\}.
\]
Finally, by using the same arguments as in the proof of Lemma \ref{lem:minoration_for_E+H}, we can prove the existence of a constant $C>0$ (depending only on $r_1,r_2$, independent of $x_0$) such that
$$
|R(r)|\leq 2C(E(r)+H(r)),\qquad\text{and thus }\frac{N'(r)}{N(r)}\geq -2C.
$$
Finally, \eqref{derivative_H(r)_Ntilde} comes from a direct calculation as in Proposition \ref{prop_almgren_rescaled}.

Therefore, at this point, we have proved the lemma for every interval $[r_1,r_2]$ with $r_2<\rbar_1$ and for every $x_0$ where $H(r)>0$. Now we need only to check that in fact $H(r)\neq 0$ for $r$ small, and the proof will be complete. This will be done in two more steps.

\textbf{$\Gamma$ has empty interior.} Assume not, and let $x_1\in \Gamma$ be such that $d_1:=\mathrm{dist}(x_1,\partial \Gamma)<\rbar_1$  (recall that we are assuming that $u^2+v^2$ is not identically zero in $\Omegatilde$). We have $H(r)>0$ for $r\in(d_1, d_1+\eps)$ for some small $\eps>0$. By what we have done so far $H(r)$ verifies, in $(d_1,d_1+\eps)$, the initial value problem
$$
\left\{\begin{array}{l}
H'(r)=a(r)H(r) \quad r\in (d_1,d_1+\eps)\\
H(d_1)=0,
\end{array}\right.
$$
with $a(r)={2}(N(r)-1)/{r}$, which is continuous also at $d_1$ by the monotonicity of $\Ntilde$. Then by uniqueness $H(r)\equiv 0$ for $r>d_1$, a contradiction with the definition of $d_1$.

\textbf{Definition of $\rbar$.} Finally we observe that, by \eqref{eq:limit_sys}, we have
$$
-\Delta u\leq (\omega_1 u^2 - \lambda) u < \lambda_1(B_r(x_0))u \qquad \text{ in } \Omega
$$
for small $r$, let us say for $0<r<\rbar_3$, independent of $x_0$ (indeed $\lambda_1(B_r(x_0))\rightarrow +\infty$ when $r\rightarrow 0$); an analogous inequality holds for $v$. Fixing now $\rbar<\min\{\rbar_2,\rbar_3\}$, so that all we have done so far holds, for $0<r\leq\rbar$ we must have $H_{x_0}(r)\neq 0$ for every $x_0$. Otherwise, for some $x_1\in \Gamma$, we would have $u,v=0$ on $\partial B_r(x_1)$. This, together with the previous inequality, would give $u,v\equiv 0$ in $B_r(x_1)$, a contradiction since $\G$ has empty interior.
\end{proof}
The previous lemma immediately provides some estimates of $N$ for small $r$.
\begin{lemma}\label{rem:alm_fin_1}
Under the previous notations, for every $x_0\in\G$, $0<r\leq\bar r$,
\[
N(0^+)\geq 2\qquad\text{and thus }\quad N(r)\geq 2 e^{-C r^2}.
\]
\end{lemma}
\begin{proof}
First of all, the limit exists finite because of the monotonicity of $\tilde N(r)=e^{Cr^2}N(r)$, and $N(0^+)=\tilde N(0^+)$. 
Let us assume by contradiction that, for some $x_0$, $N(0^+)<2$. As a consequence there exist $r^*<\rbar$ and $\eps>0$ such that, for $0<r<r^*$, we have $N(r)\leq 2-\varepsilon$. Integrating \eqref{derivative_H(r)_Ntilde} between $r$ and $r^*$, we obtain
$$
\frac{H(r^*)}{H(r)}\leq \left(\frac{r^*}{r}\right)^{2(1-\varepsilon)}.
$$
This and the fact that $u,v$ are $\alpha$--H\"older continuous for every $\a\in (0,1)$ implies $C(r^*) r^{2(1-\varepsilon)}\leq H(r)\leq C' r^{2\alpha}$ for every $\a\in (0,1)$, a contradiction.
%
%
\end{proof}
\begin{rem}\label{rem:alm_fin_2}
We recall that, for any fixed $0<r<\rbar$, the maps
\[
x_0\mapsto E_{x_0}(r),\quad x_0\mapsto H_{x_0}(r)\qquad\text{are continuous in }\overline{\G}.
\]
As a consequence, for $\rbar$ as in the previous lemma, we deduce the existence of a constants $C_1$, $C_2$, not depending on $x_0$, such that
\[
0<C_1\leq H_{x_0}(\rbar)\leq C_2 \qquad\text{for every }x_0\in\overline{\G},
\]
indeed $H$ is lower--bounded as a result of the fact that $\Gamma$ has an empty interior (Remark \ref{rem_Gamma_has_empty_interior}).
Thus also
\[
\Ntilde_{x_0}(r)\leq \Ntilde_{x_0}(\bar{r})= \frac{e^{C \bar{r}^2}\left(E_{x_0}(r)+H_{x_0}(r)\right)}{H_{x_0}(r)}\leq C_3 \qquad\text{for every }x_0\in\overline{\G},\ 0<r\leq\rbar,
\]
whit $C_3$ not depending on $x_0$.
\end{rem}
\begin{lemma}\label{lemma_previous_estimates_for_theorem_lipschitz}
Under the previous notations there exists a constant $C>0$, not depending on $x_0$ and $r$, such that
\[
\frac{1}{r^N}\int_{B_r(x_0)}\left\{|\nabla u|^2+|\nabla v|^2\right\}\leq C \qquad\text{for every }x_0\in\overline{\G},\ 0<r\leq\rbar.
\]
\end{lemma}
\begin{proof}
By Lemma \ref{lem:minoration_for_E+H} and the definitions of $N$, $\tilde N$, we know that
\[
\frac{1}{r^{N-2}}\int_{B_r(x_0)}\left\{|\nabla u|^2+|\nabla v|^2\right\}\leq 2(E(r)+H(r))=2e^{-Cr^2}\tilde N(r)H(r)\leq 2\tilde N(\bar r)H(r),
\]
and thus
\begin{equation}\label{eq:lip_quasi}
\frac{1}{r^{N}}\int_{B_r(x_0)}\left\{|\nabla u|^2+|\nabla v|^2\right\}\leq 2C_3\,\frac{H(r)}{r^{2}},
\end{equation}
where $C_3$, not depending on $r$ and $x_0$, is as in Remark \ref{rem:alm_fin_2}. In order to estimate the right hand side above, we can use \eqref{derivative_H(r)_Ntilde} to write
\[
\frac{H(\bar r)}{\bar r^2}-\frac{H(r)}{r^2}=\int_r^{\bar r}\frac{d}{d\rho}\frac{H(\rho)}{\rho^2}\,d\rho=\int_r^{\bar r}\frac{2}{\rho}\left(N(\rho)-2\right)\,d\rho,
\]
that, taking into account Lemma \ref{rem:alm_fin_1} and Remark \ref{rem:alm_fin_2}, implies
\[
\frac{H(r)}{r^2}\leq \frac{C_2}{\bar r^2}+\int_0^{\bar r}\frac{4}{\rho}\left(1-e^{-C\rho^2}\right)\,d\rho\leq C',
\]
not depending on $x_0$ and $r$. Substituting into \eqref{eq:lip_quasi}, the lemma is proved. 
\end{proof}
Finally, we are ready to prove the local Lipschitz regularity of the limiting profile.
\begin{proof}[Proof of Proposition \ref{prop:lip}] Let us assume by contradiction that $(u,v)$ is not Lipschitz continuous in $\Omegatilde$ (we follow some of the ideas of the proof of Theorem 5.1 in \cite{ctv2}, to which we refer for more details). Then there exists $\{x_n\}\subset\Omegatilde$, $r_n\to0$ such that
\begin{equation}\label{lipschitz_contradiction_argument}
\lim_{n\to +\infty}\frac{1}{r_n^N}\int_{B_{r_n}(x_n)}(|\nabla u|^2+|\nabla v|^2)=+\infty.
\end{equation}
We claim that \eqref{lipschitz_contradiction_argument} holds also for a different choice of the centers $y_n\in \overline{\Gamma}$ instead of $x_n$. This will contradict Lemma \ref{lemma_previous_estimates_for_theorem_lipschitz} and prove the proposition.

Clearly $d(x_n,\Gamma)\to0$ ($u$ and $v$ solve \eqref{eq:limit_sys} where they are positive), hence, up to a subsequence, we can assume the existence of $\xo \in \overline{\Gamma}$ such that $x_n\to\xo$. Let us start by showing that (\ref{lipschitz_contradiction_argument}) holds for a choice of $\{x_n'\}$ such that $d(x_n',\Gamma)\leq K r_n$, with $K$ independent of $n$. If, up to a subsequence, $x_n\in\overline{\Gamma}$, then there is nothing to prove. Otherwise, in every set $A_n=\{x\in\Omegatilde :d(x,\Gamma)\geq r_n\}$ there holds
\begin{eqnarray*}
\left\{ \begin{array}{lll}
-\D u+\l u=\o_1 u^3 \\
-\D v+\mu v=\o_2 v^3,
\end{array} \right.
\end{eqnarray*}
and hence
$$
-\D(|\nabla u|^2) \leq 2\nabla u\cdot\nabla(-\D u) =2(3\o_1 u^2-\l)|\nabla u|^2
$$
in every $A_n$ and similarly for $v$. If we set
$$
\Phi(x)=\frac{1}{r^N}\int_{B_r(x)}(|\nabla u(y)|^2+|\nabla v(y)|^2)dy=
\frac{1}{r^N}\int_{B_r(0)}(|\nabla u(x+y)|^2+|\nabla v(x+y)|^2)dy,
$$
then we just have proved the existence of a constant $C\geq0$ (independent on $n$) such that $-\D\Phi\leq C\Phi$ in $A_n$, for every $n$. Let now $\rho$ be so small that $-\D\f\leq C\f$ admits a strictly positive solution $\f$ in $B_\rho(\xo)$ and let $n\geq\bar{n}$ such that $x_n\in B_\rho(\xo) \ \forall n$. Then on $A_{n,\rho}=A_n\cup B_\rho(\xo)$ there holds $-\mathrm{div}(\f^2\nabla\frac{\Phi}{\f})\leq 0$, and hence by the maximum principle
$$
\max_{A_{n,\rho}}\Phi \leq C' \max_{\partial A_n}\Phi.
$$
This immediately implies that (\ref{lipschitz_contradiction_argument}) holds for a choice $\{x_n'\}$ such that $d(x_n',\Gamma)\leq K r_n$. Let now $y_n\in\overline{\Gamma}$ be such that $|y_n-x_n'|= d(x_n',\Gamma)$ and define $s_n=r_n+d(x_n',\Gamma)\leq (K+1)r_n$, then there holds
$$
\frac{1}{s_n^N}\int_{B_{s_n}(y_n)}(|\nabla u|^2+|\nabla v|^2) \geq
\frac{1}{(K+1)r_n}\int_{B_{r_n}(x_n')}(|\nabla u|^2+|\nabla v|^2) \to +\infty,
$$
but this, as we just noticed, is in contradiction with Lemma \ref{lemma_previous_estimates_for_theorem_lipschitz}, and hence $(u,v)$ is Lipschitz in $\Omegatilde$.
\end{proof}
\begin{rem}\label{rem:lip_boundry_1}
Following \cite{GL}, one can see that all the Almgren--type formulae can in fact be proved in a more general setting, that is when the Laplace operator is replaced with uniformly elliptic operators of the type
\[
-Lu=-\mathrm{div}\left(A(x)\nabla u\right),
\]
where $A$ is smooth (at least $C^1$). The key ingredient is to replace the usual polar coordinates with coordinates which are polar with respect to the geodesic distance associated to $A$. Of course the energies in the Almgren's quotient  must be defined in a suitable way. We refer to \cite{GL} for further details.
\end{rem}
\begin{rem}\label{rem:lip_boundry_2}
Once suitable Almgren's formulae are settled as in the previous remark, one can treat the Lipschitz continuity of $u$ and $v$ up to $\partial\O$ in the following way: with a local change of coordinates, and hence changing the differential operator, it is possible to assume that $\partial\O$ is locally a hyperplane, and reflect $u$ and $v$ with respect to this hyperplane. It turns out that we find new functions $\tilde u$, $\tilde v$ which satisfy a new system of equations, with different differential operators, in a larger domain $\O'\supset\supset\O$. We can then prove Lipschitz regularity of $(\tilde u,\tilde v)$, locally in $\O'$, and deduce Lipschitz regularity of $(u,v)$ in $\overline{\O}$.
\end{rem}

\subsection*{Acknowledgements}

The authors wish to thank Luis Caffarelli for the many precious
suggestions and Miguel Ramos for the fruitful discussion and careful reading of the manuscript.
The second author was supported by FCT (grant SFRH/BD/28964/2006) and by
Funda\c c\~ao Calouste Gulbenkian (grant ``Est\'imulo \`a
Investiga\c c\~ao 2007'').


\begin{thebibliography}{10}

\bibitem{Alm}
{\sc Almgren, Jr., F.~J.}
\newblock Dirichlet's problem for multiple valued functions and the regularity
  of mass minimizing integral currents.
\newblock In {\em Minimal submanifolds and geodesics ({P}roc. {J}apan-{U}nited
  {S}tates {S}em., {T}okyo, 1977)}. North-Holland, Amsterdam, 1979, pp.~1--6.

\bibitem{ACF}
{\sc Alt, H.~W., Caffarelli, L.~A., and Friedman, A.}
\newblock Variational problems with two phases and their free boundaries.
\newblock {\em Trans. Amer. Math. Soc. 282}, 2 (1984), 431--461.

\bibitem{ac}
{\sc Ambrosetti, A., and Colorado, E.}
\newblock Standing waves of some coupled nonlinear {S}chr\"odinger equations.
\newblock {\em J. Lond. Math. Soc. (2) 75}, 1 (2007), 67--82.

\bibitem{caflin}
{\sc Caffarelli, L.~A., and Lin, F.-H.}
\newblock Singularly perturbed elliptic systems and multi-valued harmonic
  functions with free boundaries.
\newblock {\em J. Amer. Math. Soc. 21}, 3 (2008), 847--862.

\bibitem{cr}
{\sc Caffarelli, L.~A., and Roquejoffre, J.-M.}
\newblock Uniform {H}\"older estimates in a class of elliptic systems and
  applications to singular limits in models for diffusion flames.
\newblock {\em Arch. Ration. Mech. Anal. 183}, 3 (2007), 457--487.

\bibitem{clll}
{\sc Chang, S., Lin, C., Lin, T., and Lin, W.}
\newblock Segregated nodal domains of two-dimensional multispecies
  {B}ose-{E}instein condensates.
\newblock {\em Phys. D 196}, 3-4 (2004), 341--361.

\bibitem{ctv}
{\sc Conti, M., Terracini, S., and Verzini, G.}
\newblock Nehari's problem and competing species systems.
\newblock {\em Ann. Inst. H. Poincar\'e Anal. Non Lin\'eaire 19}, 6 (2002),
  871--888.

\bibitem{ctv2}
{\sc Conti, M., Terracini, S., and Verzini, G.}
\newblock An optimal partition problem related to nonlinear eigenvalues.
\newblock {\em J. Funct. Anal. 198}, 1 (2003), 160--196.

\bibitem{ctv4}
{\sc Conti, M., Terracini, S., and Verzini, G.}
\newblock Asymptotic estimates for the spatial segregation of competitive
  systems.
\newblock {\em Adv. Math. 195}, 2 (2005), 524--560.

\bibitem{ctv3}
{\sc Conti, M., Terracini, S., and Verzini, G.}
\newblock On a class of optimal partition problems related to the {F}u\v c\'\i
  k spectrum and to the monotonicity formulae.
\newblock {\em Calc. Var. Partial Differential Equations 22}, 1 (2005), 45--72.

\bibitem{ctv15}
{\sc Conti, M., Terracini, S., and Verzini, G.}
\newblock A variational problem for the spatial segregation of
  reaction-diffusion systems.
\newblock {\em Indiana Univ. Math. J. 54}, 3 (2005), 779--815.

\bibitem{dww}
{\sc Dancer, E., Wei, J., and Weth, T.}
\newblock A priori bounds versus multiple existence of positive solutions for a
  nonlinear {S}chr\"odinger system.
\newblock {\em preprint\/} (2007).

\bibitem{GL}
{\sc Garofalo, N., and Lin, F.-H.}
\newblock Monotonicity properties of variational integrals, {$A\sb p$} weights
  and unique continuation.
\newblock {\em Indiana Univ. Math. J. 35}, 2 (1986), 245--268.

\bibitem{GT}
{\sc Gilbarg, D., and Trudinger, N.~S.}
\newblock {\em Elliptic partial differential equations of second order},
  second~ed., vol.~224 of {\em Grundlehren der Mathematischen Wissenschaften
  [Fundamental Principles of Mathematical Sciences]}.
\newblock Springer-Verlag, Berlin, 1983.

\bibitem{mmp}
{\sc Maia, L., Montefusco, E., and Pellacci, B.}
\newblock Infinitely many nodal solutions for a weakly coupled nonlinear
  {S}chr\"odinger system.
\newblock {\em Comm. Cont. Math. 10\/} (2008), 651--669.

\bibitem{si}
{\sc Sirakov, B.}
\newblock Least energy solitary waves for a system of nonlinear {S}chr\"odinger
  equations in {$\Bbb R\sp n$}.
\newblock {\em Comm. Math. Phys. 271}, 1 (2007), 199--221.

\bibitem{TV}
{\sc Terracini, S., and Verzini, G.}
\newblock Multipulse phases in $k$--mixtures of bose--einstein condensates.
\newblock {\em Arch. Rational Mech. Anal. to appear\/}.

\bibitem{ww}
{\sc Wei, J., and Weth, T.}
\newblock Radial solutions and phase separation in a system of two coupled
  {S}chr\"odinger equations.
\newblock {\em Arch. Rational Mech. Anal. to appear\/}.

\bibitem{ww3}
{\sc Wei, J., and Weth, T.}
\newblock Asymptotic behaviour of solutions of planar elliptic systems with
  strong competition.
\newblock {\em Nonlinearity 21}, 2 (2008), 305--317.

\end{thebibliography}

\noindent\verb"b.noris@campus.unimib.it"\\
\verb"susanna.terracini@unimib.it"\\
Dipartimento di Matematica e Applicazioni, Universit\`a degli Studi
di Milano-Bicocca, via Bicocca degli Arcimboldi 8, 20126 Milano,
Italy

\noindent \verb"htavares@ptmat.fc.ul.pt"\\
University of Lisbon, CMAF, Faculty of Science, Av. Prof. Gama Pinto
2, 1649-003 Lisboa, Portugal

\noindent \verb"gianmaria.verzini@polimi.it"\\
Dipartimento di Matematica, Politecnico di Milano, p.za Leonardo da
Vinci 32,  20133 Milano, Italy
\end{document}